\documentclass[a4paper,oneside,10pt,draft,reqno]{amsart}
\usepackage{amssymb}
\usepackage[mathcal]{euscript}
\usepackage[all]{xy}

\usepackage[l2tabu,orthodox]{nag}
\usepackage[all, warning]{onlyamsmath}

\numberwithin{equation}{section}
\allowdisplaybreaks

\theoremstyle{definition}
\newtheorem{thm}{Theorem}[section]
\newtheorem{prop}[thm]{Proposition}
\newtheorem{lem}[thm]{Lemma}
\newtheorem{cor}[thm]{Corollary}
\newtheorem{dfn}[thm]{Definition}
\newtheorem{rmk}[thm]{Remark}
\newtheorem{fct}[thm]{Fact}
\newtheorem*{dfn*}{Definition}
\newtheorem*{fct*}{Fact}
\newtheorem*{rmk*}{Remark}
\newtheorem*{lem*}{Lemma}
\newtheorem*{prop*}{Proposition}
\newtheorem*{thm*}{Theorem}
\newtheorem*{cor*}{Corollary}

\newcommand{\ve}{\varepsilon}

\newcommand{\bbC}{\mathord{\mathbb{C}}}
\newcommand{\bbK}{\mathord{\mathbb{K}}}
\newcommand{\bbZ}{\mathord{\mathbb{Z}}}

\newcommand{\bfL}{\mathord{\mathbf{L}}}
\newcommand{\bfR}{\mathord{\mathbf{R}}}

\newcommand{\g}{\mathord{\mathfrak{g}}}

\newcommand{\frkm}{\mathord{\mathfrak{m}}}
\newcommand{\frkS}{\mathord{\mathfrak{S}}}

\newcommand{\catA}{\mathord{\mathcal{A}}}
\newcommand{\catDM}{\mathord{D\mathcal{M}}}
\newcommand{\catL}{\mathord{\mathcal{L}}}
\newcommand{\catM}{\mathord{\mathcal{M}}}
\newcommand{\catN}{\mathord{\mathcal{N}}}
\newcommand{\catS}{\mathord{\mathcal{S}}}
\newcommand{\catMC}{\mathord{\mathcal{MC}}}
\newcommand{\catSet}{\mathord{\mathcal{S}ets}}
\newcommand{\catDef}{\mathord{\mathcal{D}ef}}

\newcommand{\CA}{\mathord{\mathcal{CA}}}

\newcommand{\opB}{\mathord{\mathcal{B}}}
\newcommand{\opC}{\mathord{\mathcal{C}}}
\newcommand{\opF}{\mathord{\mathcal{F}}}
\newcommand{\opP}{\mathord{\mathcal{P}}}
\newcommand{\opT}{\mathord{\mathcal{T}}}
\newcommand{\opAss}{\mathord{\mathcal{A}ssoc}}
\newcommand{\opCom}{\mathord{\mathcal{C}om}}
\newcommand{\opComu}{\mathord{\mathcal{C}omu}}
\newcommand{\opEnd}{\mathop{\mathcal{E}nd}\nolimits}
\newcommand{\opcoEnd}{\mathop{co\mathcal{E}nd}\nolimits}
\newcommand{\opHom}{\mathord{\mathcal{H}om}}
\newcommand{\opLie}{\mathord{\mathcal{L}ie}}

\newcommand{\shD}{\mathord{\mathcal{D}}}
\newcommand{\shO}{\mathord{\mathcal{O}}}
\newcommand{\shV}{\mathord{\mathcal{V}}}
\newcommand{\shY}{\mathord{\mathcal{Y}}}

\newcommand{\wh}{\widehat}

\newcommand{\longhookrightarrow}{\lhook\joinrel\longrightarrow}
\newcommand{\longtwoheadrightarrow}{\relbar\joinrel\twoheadrightarrow}

\newcommand{\ccone}{\mathbin{\circ_{(1)}}}

\newcommand{\ad}{\mathop{\mathrm{ad}}\nolimits}
\newcommand{\gr}{\mathop{\mathrm{gr}}\nolimits}
\newcommand{\id}{\mathop{\mathrm{id}}\nolimits}

\newcommand{\sgn}{\mathop{\mathrm{sgn}}\nolimits}

\newcommand{\colim}{\mathop{\mathrm{colim}}\nolimits}
\newcommand{\Wedge}{\mathop{\bigwedge}\nolimits}

\newcommand{\ch}{\mathord{ch}}
\newcommand{\cl}{\mathord{cl}}
\newcommand{\reg}{\mathord{\mathrm{reg}}}

\newcommand{\Tw}{\mathop{\mathrm{Tw}}}
\newcommand{\MC}{\mathop{\mathrm{MC}}\nolimits}
\newcommand{\Aut}{\mathop{\mathrm{Aut}}\nolimits}
\newcommand{\Cok}{\mathop{\mathrm{Coker}}\nolimits}
\newcommand{\Def}{\mathop{\mathrm{Def}}\nolimits}
\newcommand{\Der}{\mathop{\mathrm{Der}}\nolimits}
\newcommand{\End}{\mathop{\mathrm{End}}\nolimits}
\newcommand{\Hom}{\mathop{\mathrm{Hom}}\nolimits}

\newcommand{\Lie}{\mathop{\mathrm{Lie}}\nolimits}
\newcommand{\Res}{\mathop{\mathrm{Res}}\nolimits}

\newcommand{\simto}{\xrightarrow{\sim}}
\newcommand{\longto}{\longrightarrow}
\newcommand{\longsimto}{\xrightarrow{\ \sim \ }}

\newcommand{\ket}[1]{\left| #1 \right>}

\begin{document}

%%%%%%%%%%%%%%%%%%%% Authors & Paper data %%%%%%%%%%%%%%%%%%%%

\title{Deformation quantization of vertex Poisson algebras}

\author{Shintarou Yanagida}
\address{Graduate School of Mathematics, Nagoya University 
Furocho, Chikusaku, Nagoya, Japan, 464-8602.}
\email{yanagida@math.nagoya-u.ac.jp}

%\subjclass[2010]{05A15}
%\date{April 12--, 2016}
\date{July 8, 2016}

%%%%% front matter %%%%%

\begin{abstract}
We introduce dg Lie algebras controlling the deformations of 
vertex algebras and vertex Poisson algebras,
utilizing the notion of operadic dg Lie algebra and 
the theory of  chiral algebra.
In terms of those dg Lie algebras, 
we formulate the deformation quantization problem of 
vertex Poisson algebras to vertex algebras.
\end{abstract}

\maketitle
\tableofcontents

%%%%%%%%%%%%%%%%%%%%%%%%%%%%%%%%%%%%%%%%%%%%%%%%%%%%%%%%%%%%%%%%%%%%%%
%%%%%%%%%%%%%%%%%%%%%%%%%%%%%%%%%%%%%%%%%%%%%%%%%%%%%%%%%%%%%%%%%%%%%%
%%%%%%%%%%%%%%%%%%%%%%%%%%%%%%%%%%%%%%%%%%%%%%%%%%%%%%%%%%%%%%%%%%%%%%
\setcounter{section}{-1}
\section{Introduction}

%%%%%%%%%%%%%%%%%%%%%%%%%%%%%%%%%%%%%%%%%%%%%%%%%%%%%%%%%%%%%%%%%%%%%%
%%%%%%%%%%%%%%%%%%%%%%%%%%%%%%%%%%%%%%%%%%%%%%%%%%%%%%%%%%%%%%%%%%%%%%
\subsection{Vertex Poisson algebra and deformation problem}

The motivation of this note is a formulation of 
deformation quantization problem of vertex Poisson algebra.
In order to state our deformation problem,
let us begin with the recollection of vertex Poisson algebra.

%%%%%%%%%%%%%%%%%%%%%%%%%%%%%%%%%%%%%%%%%%%%%%%%%%%%%%%%%%%%%%%%%%%%%%
\subsubsection{Vertex Poisson algebra}

Let us recall the notion of vertex Poisson algebra 
following \cite[Chap.\ 16]{FBZ}.
We will use the notation $(V,\ket{0},T,Y)$ for a vertex algebra 
following  \cite[Chap.\ 1]{FBZ},
and work over $\bbC$ in this introduction.

A \emph{vertex Lie algebra} is a triple $(L, T, Y_{-})$ 
consisting of 
\begin{itemize}
\item
a vector space $L$,
\item
a linear endomorphism $T \in \End(L)$
\item
a series of endomorphisms
$Y_{-}(A,z) = \sum_{n\ge0}A_{(n)} z^{-n-1} \in 
 \End(L) \otimes z^{-1} \bbC[[z^{-1}]]$
for each $A \in L$
\end{itemize}
satisfying some conditions which we omit.
%such that
%\begin{enumerate}
%\item 
%$Y_{-}(T A,z) = \partial_z Y_{-}(A,z)$
%\item
%$Y_{-}(A,z) B = (e^{z T}Y_{-}(B,-z)A)_{-}$
%\item
%$[A_{(m)},Y_{-}(B,w)]
%=\sum_{n\ge0}\binom{m}{n} (w^{m-n}Y_{-}(A_{(n)}B,w))_{-}$
%\end{enumerate}
%for any $A,B \in L$.
%In the second axiom $f(z)_{-} := \sum_{n<0} f_n z^n$ 
%for a given series $f(z) = \sum_{n \in \bbZ} f_n z_n$.

The axiom of vertex Lie algebra is built 
so that the following statement holds.
For any vertex algebra  $(V,\ket{0},T,Y)$,
its polar part $(V,T,Y_{-})$ 
with $Y_{-}(A,z) := Y(A,z)_{-}$  is a vertex Lie algebra.
Here we set $f(z)_{-} := \sum_{n<0} f_n z^n$ 
for a given series $f(z) = \sum_{n \in \bbZ} f_n z^n$.

A vertex algebra $(V,\ket{0},T,Y)$ is called \emph{commutative}
if for any $A,B \in V$
$[Y(A,z),Y(B,w)]=0$.
It is known \cite{B}, \cite[Chap.\ 1]{FBZ} 
that a commutative vertex algebra $(V,\ket{0},T,Y)$ is equivalent 
to a commutative  $\bbC$-algebra $(V ,\circ)$ 
with unit $\ket{0}$ and a derivation $T$.
The equivalence is given by
\[
 A \circ B = A_{(-1)}B;\quad
 Y(A,z)B = e^{z T}A \circ B.
\]

A \emph{vertex Poisson algebra} is a quintuple 
$(V,\ket{0},T,Y_{+},Y_{-})$ such that
\begin{enumerate}
\item 
$(V,\ket{0},T,Y_{+})$ is a commutative vertex algebra, 
\item
$(V,T,Y_{-})$ is a vertex Lie algebra,
\item
all the coefficients of $Y_{-}(A,z)$ are 
derivations of the the commutative product on $V$ 
induced by $Y_{+}$.
\end{enumerate}

There is a natural construction of 
vertex Poisson algebras from vertex algebras,
which we call the \emph{limit construction}.
Assume that $(V^\hbar,Y^\hbar)$ is a flat family of 
vertex algebras over $\bbC[[\hbar]]$,
and that $V^0 := V^\hbar/\hbar V^\hbar$ is a 
commutative vertex algebras over $\bbC$
with $Y^0 := Y^\hbar \pmod{\hbar}$.
Then $V^0$ is a vertex Poisson algebras with 
\[
 Y_{-}(A,z) := \tfrac{1}{\hbar} 
 Y^\hbar(\overline{A},z)_{-} \pmod{\hbar},
\]
where $\overline{A} \in V^\hbar$ is a lift of $A \in V^0$.

Let us give two examples of limit construction of vertex Poisson algebra.
The first is $V_\infty(\g) := V_K(\g)/K^{-1}V_K(\g)$,
the limit of the affine vertex algebra $V_K(\g)$
for a finite Lie algebra $\g$.
Here we consider the level to be an indeterminate $K$,
and $V_K(\g)$ to be defined over $\bbC[K^{\pm1}]$.
It describes the Poisson structure on the space 
%$\Conn_G(D^\times)$
of connections on the trivial $G$-bundles on the punctured disc.
%$D^\times$.
The second one is $W_\infty(\g) := W_K(\g,e_{\reg})/K^{-1}W_K(\g,e_{\reg})$,
the limit of the $W$ algebra $W_K(\g,e_{\reg})$ 
associated to a finite dimensional $\g$ and the regular nilpotent element $e_{\reg}$.
It presents the Poisson structure on the space of opers.

%%%%%%%%%%%%%%%%%%%%%%%%%%%%%%%%%%%%%%%%%%%%%%%%%%%%%%%%%%%%%%%%%%%%%%
\subsubsection{Ad-hoc formulation  of our deformation problem}

Once one knows the limit construction of vertex Poisson algebra,
it is natural to ask the following question.
Given a vertex Poisson algebra $(V^0,\ket{0},T,Y_{+},Y_{-})$,
classify vertex algebras $(V^\hbar,\ket{0},T,Y^\hbar)$ 
flat over $\bbC[[\hbar]]$ 
such that
\[
 Y = Y_{+} + \hbar Y_{-} + \hbar^2 Y_2 + \cdots.
\]

This problem looks similar to the \emph{deformation quantization problem} 
of (usual) Poisson algebras.
Namely, given a Poisson algebra $(A,\circ, \{ \ \})$ over $\bbC$,
classify associative algebras $(A,*)$ over $\bbC[[\hbar]]$  
such that
\begin{align*}
a * b 
= a \circ b + \hbar \{a,b\} + \hbar^2 \alpha_2(a,b) + \cdots 
= \sum_{n\ge0} \hbar^n \alpha_n(a,b)
\end{align*}
with $\alpha_n\in \Hom(A^{\otimes 2},A)$.
Hereafter $\Hom$ means the space of $\bbC$-linear map.

%%%%%%%%%%%%%%%%%%%%%%%%%%%%%%%%%%%%%%%%%%%%%%%%%%%%%%%%%%%%%%%%%%%%%%
\subsubsection{Deformation quantization}

Now let us briefly recall the usual deformation quantization 
(see  \cite[\S1, \S3]{K} for a detailed explanation). 

Given a Poisson algebra $(A,\circ, \{ \, \})$,
two deformations $(A,*_1)$ and $(A,*_2)$ 
%given by
%\[
% a *_1 b = \sum_{n\ge 0}\hbar^n \alpha_n(a,b), \quad
% a *_2 b = \sum_{n\ge 0}\hbar^n \beta_n(a,b)
%\]
are called \emph{equivalent} if 
there exists 
$\varphi=\sum_{n\ge0}
 \hbar^n \varphi_n \in \End(A) [[\hbar]]$
such that
$\varphi(a *_1 b) = \varphi(a) *_2 \varphi(b)$.

The equivalent class of deformations are 
described by the \emph{Hochschild complex}
$C^\bullet(A,A) = (\oplus_{n\ge0} C^n(A,A), d)$,
and the \emph{Hochschild cohomology}
$H^\bullet(A,A)$.
Recall that the Hochschild complex is given by 
\begin{align}
\label{eq:Hochschild}
&C^n(A,A) := \Hom(A^{\otimes n}, A), 
\\
\nonumber
&d f(a_0, \ldots, a_n)
 := a_0 \circ f(a_1, \ldots, a_n) 
    +\sum_{i=1}^n (-1)^i 
    f(a_0, \ldots, (a_{i-1} \circ a_i), \ldots,a_n) 
    + (-1)^{n+1} f(a_0, \ldots, a_{n-1}) \circ a_n.
\end{align}
$H^\bullet(A,A)$ is the cohomology of this complex.
As for the deformations of $(A,\circ,\{\,\})$, we have
\begin{itemize}
\item 
the equivalence class of $\alpha_1 = \{ \, \}$ is 
an element of $H^2(A,A)$,
\item
using the Gerstenhaber bracket $[\ ]'$ given by
\[
 \tfrac{1}{2} [\alpha_i,\alpha_j]'(a,  b, c)
 := \alpha_i(\alpha_j(a, b), c)
  - \alpha_i(a, \alpha_j(b,c)),
\]
the associativity of $*$ is rewritten as 
\[
 d \alpha_m + \dfrac{1}{2}\sum_{i+j=m}
 [\alpha_i,\alpha_j]' = 0.
\]
\end{itemize}

These properties can be expressed in terms of \emph{dg Lie algebra}.
The equations  
$d \alpha_m + \tfrac{1}{2}\sum_{i+j=m}
 [\alpha_i,\alpha_j]' = 0$
can be rewritten as the \emph{Maurer-Cartan equation}
\begin{equation}\label{eq:MC}
  d \alpha + [\alpha,\alpha] = 0,\quad 
 \alpha = \sum_n  \alpha_n \in \g^1
\end{equation}
of the dg Lie algebra $\g=(C^\bullet(A;A),[\ ],d)$ 
associated to the Hochschild complex.
Here the grading is 
\[
\g = \oplus_{n \ge -1} \g^n,\quad
 \g^n := C^{n+1}(A,A)= \Hom(A^{\otimes(n+1)},A).
\]
The Lie bracket $[ \, ]$ is given by 
\begin{align*}
&[\alpha,\beta]:=
 \alpha \circ \beta - (-1)^{|\alpha| \cdot |\beta|} \beta \circ \alpha,
\\
&(\alpha \circ \beta) (a_0, \ldots, a_{|\alpha| +|\beta|})
:= \sum_{r=0}^{|\alpha|} (-1)^{r |\beta|} 
 \alpha (a_0, \ldots, a_{r-1},
  \beta(a_{r}, \ldots, a_{r+|\beta|}),
  a_{r+|\beta|+1}, \ldots, a_{|\alpha| +|\beta|} ).
\end{align*}

Thus a deformation quantization of a Poisson algebra 
is the problem to find and classify solutions 
$\alpha = \sum_{n\ge0} \alpha_n$ of 
Maurer-Cartan equation of the dg Lie algebra $C^\bullet(A,A)$ 
with $\alpha_0$ and $\alpha_1$ equal to the given $\circ$ and $\{\ \}$.

%%%%%%%%%%%%%%%%%%%%%%%%%%%%%%%%%%%%%%%%%%%%%%%%%%%%%%%%%%%%%%%%%%%%%%
\subsubsection{Main problem: the dg Lie algebra for our deformation problem}

Now going back to our situation,
it is natural to ask 
what the dg Lie algebra controlling deformations of vertex Poisson algebra is.

One may find two hints in the literature.
\begin{itemize}
\item 
According to the \emph{theory of operad},
one can construct a  dg Lie algebra 
controlling the deformations of $\opP$-algebras 
for any Koszul operad $\opP$.
\item
According to the \emph{theory of chiral algebra} by Beilinson and Drinfeld,
vertex algebras and vertex Poisson algebras can be formulated in terms of operads.
\end{itemize}

%%%%%%%%%%%%%%%%%%%%%%%%%%%%%%%%%%%%%%%%%%%%%%%%%%%%%%%%%%%%%%%%%%%%%%
%%%%%%%%%%%%%%%%%%%%%%%%%%%%%%%%%%%%%%%%%%%%%%%%%%%%%%%%%%%%%%%%%%%%%%
\subsection{Recollection of Operads}

Here we  briefly explain notations 
and some basic facts of operads.
We refer \cite[Chap.\ 5--6]{LV} for the fundamentals of operads.
We denote by $\frkS_n$ the $n$-th symmetric group. 
%and work over $\bbC$.

%%%%%%%%%%%%%%%%%%%%%%%%%%%%%%%%%%%%%%%%%%%%%%%%%%%%%%%%%%%%%%%%%%%%%%
\subsubsection{Basic notions}

An \emph{$\frkS$-module} is a series $M = \{M(n)\}_{n\ge0}$ of 
right $\frkS_n$-modules.
We also denote $M$ as $M=\oplus_{n\ge0} M(n)$.

The first example of $\frkS$-module is the \emph{identity $\frkS$-module}
$I= 0 \oplus \bbC \oplus  0 \oplus 0 \oplus \cdots$ 

For two  $\frkS$-modules $M$ and $N$,  define 
another $\frkS$-module $M \circ N$ by
$M \circ N := \oplus_n M(n) \otimes_{\frkS_n} N^{\otimes n}$.

A \emph{morphism} of $\frkS$-modules 
is a series of $\frkS_n$-module homomorphisms.

An $\frkS$-module $\opP$ with $\opP(0)=0$ is called \emph{reduced}.
A reduced $\frkS$-module is sometimes denoted as $\opP=\oplus_{n\ge1}\opP(n)$.

An \emph{operad} is a triple 
\[
 \opP=(\opP, \gamma,\eta)
\] 
of 
\begin{itemize}
\item 
an $\frkS$-module $\opP=\oplus_{n\ge0} \opP(n)$
consisting of spaces of $n$-ary operations,
\item
an $\frkS$-module morphism 
$\gamma: \opP \circ \opP \to \opP$
called the \emph{composition map},
\item
an $\frkS$-module morphism $\eta: I \to \opP$
called the unit.
\end{itemize}
satisfying some compatibility conditions.
A \emph{morphism of operads} is defined to be 
a morphism of underlying $\frkS$-modules 
which are compatible with $\gamma$'s  and $\eta$'s.

Let us denote by
$\opAss$, $\opCom$ and $\opLie$
the operads of commutative algebras,
(non-commutative) associative algebras
and  Lie algebras respectively.
For each of these operads, denoted as $\opP$,
we have an element $\mu \in \opP(2)$ 
corresponding to the binary operation 
(associative product, commutative product and Lie bracket).

We also have 
the \emph{operad of endomorphisms} on a vector space $V$.
Set the $\frkS$-module $\opEnd_V$ by 
\[
  \opEnd_V:=\oplus_{n\ge0} \opEnd_V(n),\quad
  \opEnd_V(n) := \Hom_{\bbC}(V^{\otimes n},V).
\]
Then the composition of linear endomorphisms 
gives $\opEnd_V$ a natural structure of operad.
%$\gamma$ is the usual composition of linear maps.

Now for an operad $\opP$ and a vector space $V$,
a \emph{$\opP$-algebra structure} on $V$ 
is an operad  morphism
\[
 \opP \longto \opEnd_V.
\]
After a moment thought, one finds that for 
$\opP=\opAss$, $\opCom$ and $\opLie$,
an $\opP$-algebra is nothing but the usual 
associative, commutative and Lie algebra respectively.

We will also need \emph{cooperads}, dually defined as operads.
A cooperad is a triple 
\[
 \opC = (\opC, \Delta,\ve)
\] 
consisting of a $\frkS$-module $\opC$, the \emph{decomposition map}
$\Delta: \opC \to \opC \circ \opC$
and the counit map $\ve: \opC \to I$.
%For a vector space $V$, 
%we have the cooperad of endomorphisms 
%whose underlying $\frkS$-module is given by 
%\[
% \opcoEnd_V=\oplus_{n \ge0} \opcoEnd(n),\quad
% \opcoEnd_V(n) := \Hom_{\bbC}(V,V^{\otimes n}).
%\]

%%%%%%%%%%%%%%%%%%%%%%%%%%%%%%%%%%%%%%%%%%%%%%%%%%%%%%%%%%%%%%%%%%%%%%
\subsubsection{Koszul dual (co)operads}

Let us denote by  $\opF(E)$ 
the free operad of an $\frkS$-module $E$.
We will give a brief account in \S\ref{sss:bar} and \S\ref{sss:Koszul-dual}.
It has a \emph{weight} grading $\opF(E)= \oplus_{d\ge0} \opF(E)^{(d)}$.
In a dual way, we have the \emph{free cooperad} 
$\opF^c(E)$ of an $\frkS$-module $E$. 
%Its underlying $\frkS$-module is the same as $\opF(E)$.

Recall the notion of \emph{quadratic operad}.
A quadratic data $(E,R)$ is a pair of 
an $\frkS$-module $E$ and a sub-$\frkS$-module $R \subset \opF(E)^{(2)}$.
Now the quotient 
\[
 \opP(E,R):=\opF(E)/(R)
\]
has a structure of operad, and called 
the \emph{quadratic operad} for $(E,R)$.
The operads $\opAss$, $\opCom$ and $\opLie$ are standard examples.
In a dual way,
we have the \emph{quadratic cooperad}  $\opC(E,R)$ 
of the quadratic data $(E,R)$.

The notions mentioned so far have \emph{dg version}.
Namely, a dg $\frkS$-module is a series of complexes 
with right action of $\frkS_n$.
Similarly, one can define a dg (co)operad, 
a dg (co)operad of endomorphisms 
and and a dg free (co)operad.

Let us denote by $s$ the shift of a complex in the following way. 
\[
 (s M)_p = M_{p-1}.
\]
Then the \emph{Koszul dual cooperad} 
of a quadratic operad $\opP=\opP(E,R)$ is given by 
\[
 \opP^{c!} := \opC(s E, s^2 R).
\]
$\opP^{c!}$ %and $\opP^!$ 
has a weight grading similarly 
as the free operad $\opF(E)$.

%Although it is not necessary in our discussion,
%let us recall that the Koszul dual operad of $\opP$ is given by
%$\opP^{!} := (\opcoEnd_{s \bbC} \otimes \opP^{c!})^*$,
%where $\opcoEnd_{V}$ denotes the cooperad of endomorphisms on $V$ 
%whose underlying $\frkS$-module is given by 
%$\opcoEnd_V=\oplus_{n \ge0} \Hom_{\bbC}(V,V^{\otimes n})$.
%%$\opP^{c!}$ and $\opP^!$ have weight gradings similarly 
%%as the free operad $\opF(E)$.
%As is well known, we have 
%$\opAss^! \simeq \opAss$, 
%$\opCom^! \simeq \opLie$ and 
%$\opLie^! \simeq \opCom$.

%%%%%%%%%%%%%%%%%%%%%%%%%%%%%%%%%%%%%%%%%%%%%%%%%%%%%%%%%%%%%%%%%%%%%%
\subsubsection{Convolution dg Lie algebra}
\label{subsubsec:convol-dgla}

We recall the convolution dg Lie algebra following 
\cite[Chap.\ 6]{LV}.

There is a good class of quadratic operads called 
\emph{Koszul operads}.
%One of the standard definitions is that 
%the natural projection $\Omega \opP^{c!} \twoheadrightarrow \opP$ from 
%the cobar resolution $\Omega \opP^{c!}$ of $\opP$
%is a quasi-isomorphism.
The operads $\opAss$, $\opCom$ and $\opLie$ are Koszul.

\begin{fct}[{\cite[Chap.\ 6]{LV}}]
\label{fct:convol}
For a Koszul operad $\opP$ and a vector space $V$,
the vector space 
\[
 \g \equiv \g_{\opP,V} \equiv \Hom_{\frkS}(\opP^{c!},\opEnd_V)
 := \oplus_{n\ge0} \Hom_{\frkS_n}(\opP^{c!}(n),\opEnd_V(n))
\]
has a natural structure of dg Lie algebra
$(\g, [ \ ], \partial)$. 
\end{fct}

It is called the \emph{convolution dg Lie algebra}.
Note that it has a double grading.
One is  the grading 
$\g= \oplus_{n} \g^n$ as a complex, 
and the other is the weight grading $\g= \oplus_{d\ge0} \g^{(d)}$
induced by that on $\opP^{c!}$.

Now recall the Maurer-Cartan equation \eqref{eq:MC} 
which is defined for  any dg Lie algebra $\g=\oplus_{n}\g^n$.
Define 
\[
 \Tw(\g) := 
 \{ \text{degree $n=-1$ solutions of the Maurer-Cartan equation of $\g$} \}.
\]

\begin{fct}[{\cite[Proposition 10.1.4]{LV}}]
\label{fct:P-alg}
For a Koszul operad $\opP$ and a vector space $V$, 
\begin{align*}
\{ \text{ $\opP$-algebra structures on $V$} \} 
\ \stackrel{\ \ 1:1 \ \ }{\longleftarrow \joinrel \longrightarrow} \ 
 \MC(\g_{\opP,V}) := 
 \{ \text{weight $d=1$ elements in $\Tw(\g_{\opP,V})$} \}
\end{align*}
\end{fct}

Once a weight $1$ solution $\mu$ of the Maurer-Cartan equation is given,
we have a twisted dg Lie algebra
\[
  \g^\mu_{\opP,V}
  := (\Hom(\opP^{c!},\opEnd_V), [ \ ], \partial^\mu := \partial + [\mu,-]).
\]
and it encodes the deformation of $\opP$-algebra structure $\mu$.
\begin{align*}
&\MC(\g^\mu_{\opP,V})
  \stackrel{\ 1:1 \ }{\longleftarrow \joinrel \longrightarrow} 
 \{\text{$\opP$-algebra structures deforming $\mu$} \}. 
\end{align*}
For example,
if $\opP = \opAss$, then $\mu \in \MC(\g_{\opAss,V})$ 
is an associative product on $V$,
and $\g^\mu_{\opAss,V}$ coincides with the Hochschild complex 
\eqref{eq:Hochschild}
%$(C^\bullet(A,A),[\ ],d)$ 
for $A=(V,\mu)$.
%As an another example, 
%If $\opP = \opLie$,
%then $\mu \in \MC(\g_{\opLie,V})$ 
%is a Lie bracket on $V$,
%and $\g^\mu_{\opLie,V}$ coincides with 
%the Chevalley complex $(C^\bullet(L,L),d)$ 
%of the Lie algebra $L=(V,\mu)$ 
%and the Nijenhuis-Richardson bracket $[ \ ]$ on the Chevalley complex.

%%%%%%%%%%%%%%%%%%%%%%%%%%%%%%%%%%%%%%%%%%%%%%%%%%%%%%%%%%%%%%%%%%%%%%
%%%%%%%%%%%%%%%%%%%%%%%%%%%%%%%%%%%%%%%%%%%%%%%%%%%%%%%%%%%%%%%%%%%%%%
\subsection{Chiral dg Lie algebra}

The theory of \emph{chiral algebras} by
Beilinson and Drinfeld \cite{BD}
is an operadic formulation of vertex algebras.
One can apply to it the general construction of convolution dg Lie algebra
in the previous subsection.
The obtained dg Lie algebra is what we look for.

%%%%%%%%%%%%%%%%%%%%%%%%%%%%%%%%%%%%%%%%%%%%%%%%%%%%%%%%%%%%%%%%%%%%%%
\subsubsection{Chiral algebras}

Here we briefly spell out what a chiral algebra is,
and defer a detailed explanation to \S \ref{sect:CA}.

Let $X$ be a smooth curve.
$\catM(X)$ denotes the category of right $\shD_X$-modules 
(quasi-coherent as $\shO_X$-modules) on $X$.
%on a smooth variety $X$.
%For a morphism $f:X \to Y$ of smooth varieties,
%denote the standard functors as 
%\[
% f_* : D \catM(X) \longto D \catM(Y),\quad
% f^!:  D \catM(Y) \longto D \catM(X).
%\]
%For $f:X \to Y$ a LCI over $Y$ of pure dim.\ $d$,
%\[
% f^* := f^![-d]: D \catM(Y) \longto D \catM(X),\quad
% \text{in fact} \quad \catM(Y) \longto \catM(X).
%\]
For $n \in \bbZ_{\ge1}$ denote by
$\Delta^{(n)}: X \longhookrightarrow X^n$ the diagonal embedding and 
$j^{(n)}: U^{(n)}:=\{(x_i) \in X^n \mid x_i \neq x_j \ (\forall \, i \neq j)\}
 \longhookrightarrow X^n$
the complement of diagonal divisors.

For $M \in \catM(X)$, 
consider a $\frkS$-module 
$\opEnd^{\ch}_M = \oplus_n \opEnd^{\ch}_M(n)$ given by 
\[
 \opEnd^{\ch}_M(n) := 
 \Hom_{\catM(X^n)}(j^{(n)}_* j^{(n) \, *} M^{\boxtimes n}, 
 \Delta^{(n)}_* M)
\]
It has an operad structure, which we call 
the \emph{chiral operad} on $M$.

%Remark:
%\begin{align*}
%j^{(2)}_* j^{(2) \, *} M^{\boxtimes 2} \simeq 
%M^{\boxtimes 2}(\infty \Delta_{\text{diag}}),
%%\Delta^{(2)}_* M \simeq M^{\boxtimes 2}
%\end{align*}

A \emph{chiral algebra} structure (without unit) 
on $M \in \catM(X)$ is an operad morphism
\[
 \varphi: \opLie \longto \opEnd^{\ch}_M.
\]
Denoting $\mu_{\Lie} \in \opLie(2)$ the binary operation 
corresponding to the Lie bracket,
we call Its image $\varphi(\mu_{\Lie}) \in \opEnd^{\ch}_M(2)$
the \emph{chiral bracket}.

Now we recall the relation between vertex algebras 
and chiral algebras.
By \cite[Chap.\ 5]{FBZ},
one can construct from a vertex algebra $V$ a locally free sheaf $\shV$ 
on a smooth curve $X$.
If $V$ is quasi-conformal, then $\shV$ is a left $\shD_X$-module. 
Denote by 
\[
 \shY^2: j^{(2)}_* j^{(2) *} \shV^{\boxtimes 2} \longto \Delta^{(2)}_*(\shV)
\] 
the morphism of left $\shD$-modules induced by $Y$.
Locally it is given by 
\[
 \shY^2_x(f(z,w) A \boxtimes B) = f(z,w)Y(A,z-w)B \pmod{ V[[z,w]] }. 
\]
Also denote by $\omega_X$ the canonical sheaf on $X$.
Given a left $\shD_X$-module $M$,
one has a right $\shD_X$-module 
\[
 M^r := M \otimes_{\shO_X} \omega_X.
\]

\begin{fct}[{\cite{BD},  \cite[Chap.\ 19]{FBZ}}]
\label{fct:VA=CA}
For a quasi-conformal vertex algebra $V$,
the right $\shD$-module $\shV^r := \shV \otimes \omega_X$
has a structure of chiral algebra.
The chiral bracket $\mu \in \opEnd^{\ch}_{\shV^r}(2)$ 
is given by $\mu=(\shY^2)^r$.
\end{fct}

%%%%%%%%%%%%%%%%%%%%%%%%%%%%%%%%%%%%%%%%%%%%%%%%%%%%%%%%%%%%%%%%%%%%%%
\subsubsection{Coisson algebras}
\label{subsubsec:intro:coisson}

In \cite[\S2.6]{BD},
an operad structure 
corresponding to a vertex Poisson algebra is defined.
we call it the \emph{compound operad} and denote it by 
$\opEnd^c_M$ for $M \in \catM(X)$.
Its underlying $\frkS$-module is given by 
\begin{align*}
 \opEnd^c_M(n) = \oplus_{S \in Q([n])} \opEnd^c_M(n)_S, \quad
 \opEnd^c_M(n)_S := 
 \Hom_{\catM(X^S)}(M^{\boxtimes S},\Delta^{(S)}_* M)
 \otimes (\otimes_{t \in S} \opLie(|\pi_S^{-1}(t)|)).
\end{align*}
Here $Q([n])$ denotes the set of equivalence classes of 
surjections $\pi_S: [n] \longtwoheadrightarrow S$ from
the set $[n]=\{1,\ldots,n\}$,
and $\Delta^{(S)}: X \hookrightarrow X^S$ 
denotes the diagonal embedding.

A \emph{coisson algebra} structure on $M\in\catM(X)$ 
is an operad morphism
\[
 \opLie \longto \opEnd^c_M.
\]

Let us explain the relation between
coisson algebras and chiral algebras.
$\opEnd^{\ch}_M$ has a filtration $W^\bullet$ such that 
\[
 \opEnd^{\ch}_M(n) = W^0 \supset W^{-1} \supset \cdots \supset W^{-n} \supset 
 W^{-n-1}=0,
\]
and we have an inclusion of operads
\[
 \gr_{W} \opEnd^{\ch}_M \longhookrightarrow \opEnd^c_M.
\]
The filtration comes from the Cousin complex of $\omega_X$.

As a corollary, 
given a family $A_t$ of chiral algebras flat over $\bbC[[t]]$,
$A_0 := A_t/t A_t$ has a structure of coisson algebra.
Since the bundle of the vertex Poisson algebra 
yields a coisson algebra,
the limit construction of vertex Poisson algebra from 
a quasi-conformal vertex algebra,
as mentioned previously, is a special case of coisson to chiral limit.

%%%%%%%%%%%%%%%%%%%%%%%%%%%%%%%%%%%%%%%%%%%%%%%%%%%%%%%%%%%%%%%%%%%%%%
%%%%%%%%%%%%%%%%%%%%%%%%%%%%%%%%%%%%%%%%%%%%%%%%%%%%%%%%%%%%%%%%%%%%%%
\subsection{Chiral dg Lie algebra}

Finally we can introduce our main object. 
As in the previous subsection,
let $X$ denote a smooth curve. %and 
%$\catM(X)$ denote the category of right $\shD_X$-modules.

%%%%%%%%%%%%%%%%%%%%%%%%%%%%%%%%%%%%%%%%%%%%%%%%%%%%%%%%%%%%%%%%%%%%%%
%%%%%%%%%%%%%%%%%%%%%%%%%%%%%%%%%%%%%%%%%%%%%%%%%%%%%%%%%%%%%%%%%%%%%%
\subsubsection{Definition of our dg Lie algebra}

Similarly as in Fact \ref{fct:convol}, 
we have the following construction.

\begin{prop*}
For a right $\shD_X$-module $M$, the $\frkS$-module
\[
 \g^{\ch}_{M} := \Hom(\opLie^{c!},\opEnd^{\ch}_M) 
\]
has a structure of dg Lie algebra 
$(\g^{\ch}_{M},[ \ ], \partial=0)$.
\end{prop*}

\begin{dfn*}
We call $\g^{\ch}_{M}$ the \emph{chiral Lie algebra}.
\end{dfn*}

Then similarly as in Fact \ref{fct:P-alg},
we have the following description of chiral algebra structure.

\begin{prop*}
For a right $\shD_X$-module $M$, 
\[
 \{\text{chiral algebra structures on $M$}\} 
 \ \stackrel{\ \ 1:1 \ \ }{\longleftarrow \joinrel \longrightarrow} \ 
 \MC(\g^{\ch}_M).
\]
\end{prop*}

Given $\alpha \in  \MC(\g^{\ch}_M)$,
we can twist the graded Lie algebra $\g^{\ch}_M$ 
as follows.

\begin{dfn*}
The following has a structure of dg Lie algebra,
which we call the \emph{chiral dg Lie algebra}.
\[
 \g^{\ch,\alpha}_M := (\g^{\ch}_{M},[ \ ], \partial_\alpha),
 \quad
 \partial_\alpha := \partial +[\alpha,-] = [\alpha,-].
\]
\end{dfn*}

The standard deformation theory says

\begin{prop*}
Given $\alpha \in  \MC(\g^{\ch}_M)$,
we have 
\[ 
 \{\text{chiral algebra structures on $M$ deforming $\alpha$}\} 
 \ \stackrel{\ \ 1:1 \ \ }{\longleftarrow \joinrel \longrightarrow} \ 
 \MC(\g^{\ch,\alpha}_M).
\]
\end{prop*}

We have a similar argument for coisson algebras.
One can construct the \emph{coisson dg Lie algebra} 
whose underlying $\frkS$-module is given by 
\[
 \g^{c}_{M} := \Hom(\opLie^{c!},\opEnd^{c}_M). 
\]
A coisson algebra structure corresponds to a weight $1$ 
element of $\Tw(\g^{c}_{M})$ bijectively.

%%%%%%%%%%%%%%%%%%%%%%%%%%%%%%%%%%%%%%%%%%%%%%%%%%%%%%%%%%%%%%%%%%%%%%
\subsubsection{Deformation problem revisited}

By the arguments in \S\ref{subsubsec:intro:coisson} of coisson algebras,
we have morphisms of operads
\[
 \opEnd^{\ch}_M \longtwoheadrightarrow
 \gr\opEnd^{\ch}_M \longhookrightarrow 
 \opEnd^{c}_M.
\]

The first theorem in this note is 

\begin{thm*}
The above induces a morphism of dg Lie algebras
\[
 \g^{\ch}_M \longto \g^{c}_M
\]
and hence we have a map 
\[
 \psi: \MC(\g^{\ch}_M) \longto \MC(\g^{c}_M).
\]
\end{thm*}

See Remark \ref{rmk:dq-cdq} for the difference between our situation 
and the usual deformation quantization of associative algebras.

\begin{dfn*}
We call $\mu \in \MC(\g^{c}_M)$ 
a \emph{chiral deformation quantization} of  $\mu^{c} \in \MC(\g^{c}_M)$
if $\psi(\mu)=\mu^{c}$.
\end{dfn*}

Obviously, if the map $\psi$ is surjective,
then a chiral deformation quantization exists.
Our main theorem is 

\begin{thm*}
$\psi$ is always \emph{injective}.
\end{thm*}

Thus, a vertex Poisson algebras arising from the limit construction,
like $V_{\infty}(\g)$ and $W_{\infty}(\g,e_{\reg})$,  
has a unique chiral deformation quantization.

%%%%%%%%%%%%%%%%%%%%%%%%%%%%%%%%%%%%%%%%%%%%%%%%%%%%%%%%%%%%%%%%%%%%%%
%%%%%%%%%%%%%%%%%%%%%%%%%%%%%%%%%%%%%%%%%%%%%%%%%%%%%%%%%%%%%%%%%%%%%%
\subsection{Organization of this note}

In \S\ref{sect:CA} we give a recollection of chiral algebras 
following \cite{BD}.
We start \S\ref{subsec:d-mod} with the $*$-pseudo-tensor structure 
and $!$-tensor structure 
on the category of $\shD_X$-modules.
Using these structures,  
we introduce coisson algebras in \S\ref{subsec:cois} and 
chiral algebras in \S\ref{subsec:CA}.
In \S\ref{subsec:cl} and \S\ref{subsec:limit} 
we explain the relation 
between chiral and coisson algebras.

\S\ref{sect:conv} is the recollection of operad theory 
and the deformation theory of algebras over an operad 
following \cite{LV}.
In \S\ref{subsec:conv-dgla} 
we introduce the operadic convolution dg Lie algebra.
In \S\ref{subsec:mc-alg} 
we recall quadratic operads, Koszul dual and Koszul operads.
In \S\ref{subsec:mc-alg} 
the dg Lie algebra controlling operadic algebra structure 
is introduced.
\S\ref{subsec:2:deform} gives a brief recollection of deformation 
theory using dg Lie algebras.

\S\ref{sect:chdgla} is the main part of this note,
and we will use all the notions prepared so far.
\S\ref{subsec:calc} introduces and studies the dg Lie algebras 
controlling the deformations of chiral and coisson algebras.
In \S\ref{subsec:deform} we cast our deformation quantization problem to 
the dg Lie algebras and prove the main theorems.

%%%%%%%%%%%%%%%%%%%%%%%%%%%%%%%%%%%%%%%%%%%%%%%%%%%%%%%%%%%%%%%%%%%%%%
\subsection*{General Notations}

For a category $\catM$, $A \in \catM$ means that 
$A$ is an object of $\catM$.
%Denote by $\Hom_{\catM}(A,B)$ 
%the set of morphisms between $A, B \in \catM$.

A tensor category means a monoidal category in the sense of \cite{M}.

$\catSet$ denotes the category of arbitrary sets and maps.
For a set $I$, $|I|$ denotes its cardinality.

In the main text we will work over a fixed field $\bbK$ of 
characteristic $0$.
$\otimes$ denotes the tensor product $\otimes_{\bbK}$ 
of $\bbK$-vector spaces unless otherwise stated.

%%%%%%%%%%%%%%%%%%%%%%%%%%%%%%%%%%%%%%%%%%%%%%%%%%%%%%%%%%%%%%%%%%%%%%
%%%%%%%%%%%%%%%%%%%%%%%%%%%%%%%%%%%%%%%%%%%%%%%%%%%%%%%%%%%%%%%%%%%%%%
%%%%%%%%%%%%%%%%%%%%%%%%%%%%%%%%%%%%%%%%%%%%%%%%%%%%%%%%%%%%%%%%%%%%%%
\section{Recollection of chiral algebras and coisson algebras}
\label{sect:CA}

This section gives a recollection of chiral algebras and coisson algebras
following \cite{BD}.

Beilinson and Drinfeld started \cite{BD} with a large account 
of the general theory of pseudo-tensor category.
It is equivalent to the notion of colored operad, 
and roughly speaking,
it is a category whose sets of morphisms have composition rules.
A pseudo-tensor structure with one object is nothing but an operad.

We avoid copying their argument and specialize it to 
the case of $\shD$-modules.
The main purpose of this section is to introduce two operads,
which we name \emph{chiral operad} and \emph{coisson operad}.
These are restricted versions of chiral and coisson pseudo-tensor structures 
originally introduced in \cite[\S1.4, \S2.2, \S3.1]{BD},
and encode the operadic structure of vertex algebras and 
vertex Poisson algebras.

%%%%%%%%%%%%%%%%%%%%%%%%%%%%%%%%%%%%%%%%%%%%%%%%%%%%%%%%%%%%%%%%%%%%%%
%%%%%%%%%%%%%%%%%%%%%%%%%%%%%%%%%%%%%%%%%%%%%%%%%%%%%%%%%%%%%%%%%%%%%%
\subsection{$\shD$-module category and (pseudo)-tensor structures}
\label{subsec:d-mod}

We follow \cite[\S2.2]{BD}.

Let $X$ be a smooth scheme over a fixed field $\bbK$ of 
characteristic $0$.
$\shO_X$, $\shD_X$ and $\Theta_X$ denote the structure sheaf,
the sheaf of differential operators,
and the sheaf of vector fields on $X$ respectively.
$\catM_{\shO}(X)$ denotes the category of quasi-coherent $\shO_X$-modules.
We will also use the simplified symbols $\shO:=\shO_X$ and  $\shD := \shD_X$.
Denote by $\catM(X)$ the category of right $\shD$-modules  on $X$
(more precisely,  sheaves of right $\shD_X$-modules 
 which are quasi-coherent as $\shO_X$-modules).
Similarly denote by $\catM^\ell(X)$ the category of left $\shD$-modules.

%%%%%%%%%%%%%%%%%%%%%%%%%%%%%%%%%%%%%%%%%%%%%%%%%%%%%%%%%%%%%%%%%%%%%%
\subsubsection{The $!$-tensor structure}

For $L_1,L_2 \in \catM^\ell(X)$, 
$L_1 \otimes_{\shO_X} L_2$ is naturally a left $\shD$-module.
So $\catM^\ell(X)$ is a tensor category 
with a unit object $\shO_X$.

For a right $\shD$-module $M$ and a left $\shD$-module $L$,
the sheaf $M \otimes_{\shO_X} L$ is naturally a right $\shD$-module by
\[
 (m \otimes l) \tau := m\tau \otimes l - m \otimes \tau l
\]
for $\tau \in \Theta_X \subset \shD_X$.
We will denote this right $\shD$-module by $M \otimes L$.

The canonical sheaf $\omega_X := \Omega_X^{\dim X}$ 
has a canonical right $\shD$-module structure
\[
 \nu \tau = -\Lie_{\tau}(\nu)
\]
for $\nu \in \omega_X$,
where $\Lie_\tau$ denotes the Lie derivative.
Then we have the standard equivalence 
\[
 \catM^\ell(X) \to \catM(X),\quad 
 L \longmapsto L^r := \omega_X \otimes L.
\]
The inverse is give by 
%\begin{equation}\label{eq:stdeq:M-Ml}
\[
 M \longmapsto M^\ell := M \otimes \omega_X^{-1}.
\]
%\end{equation}

Pulling back the tensor structure on $\catM^\ell(X)$ 
by this equivalence,
we have a tensor structure on $\catM(X)$.
Namely, 
\[
 M \otimes^! N := M^\ell \otimes_{\shO_X} N.
\]
The canonical sheaf $\omega_X$ is a unit object for $\otimes^!$.

%%%%%%%%%%%%%%%%%%%%%%%%%%%%%%%%%%%%%%%%%%%%%%%%%%%%%%%%%%%%%%%%%%%%%%
\subsubsection{The $*$-pseudo-tensor structure}

In order to introduce the $*$-pseudo-tensor structure,
we need to recall some functors of $\shD$-modules.
Let $\catDM(X)$ denote the derived category of 
right $\shD$-modules on $X$,
and $\catDM^\ell(X)$ the derived category of left $\shD$-modules.
We have an equivalence 
\[
 \catDM^\ell(X) \longto \catDM(X), \quad
 L \longmapsto \omega_X \otimes L [\dim X].
\]
Thus $\catDM(X)$ has the original $t$-structure $\catM(X)$ 
and the other $t$-structure $\catM^\ell(X)$,
differing shifts by $\dim X$.

For a morphism $f:X \to Y$ of smooth schemes,
we have the standard derived functors 
\[
 f_*: D \catM(X) \longto D \catM(Y),\quad
 f^!: D \catM(Y) \longto D \catM(X).
\]
If $f$ is a closed embedding, then $f_*$ is exact with respect to 
the $t$-structure $\catM(X)$, and its right adjoint $f^!$ is left exact.
They define the equivalence (Kashiwara's lemma)
\[
 f_*: \catM(X) \longleftrightarrow \catM(Y)_X  :f^!.
\]
Here $\catM(Y)_X \subset \catM(Y)$ is the full subcategory of 
$\shD_Y$-modules vanishing on $Y \setminus X$.

%Since $f^!$ is compatible with the standard pull-back functors 
%$\bfR f^!, \bfL f^*:\catDM_{\shO}(Y) \to \catDM_{\shO}(X)$ for $\shO$-modules,
%$f^!$ is right exact with respect to the $\catM^\ell$ $t$-structure .
%We denote the corresponding functor as 
%\[
% f^*: \catM^\ell(Y) \longto \catM^\ell(X).
%\]

For a finite collection $\{X_i\}_{i \in I}$ of smooth schemes,
denote by $\boxtimes$ the exterior tensor product.
Namely we have 
\[
 \prod_{i \in I} \catM(X) \longto \catM(\prod_{i \in I}X_i),\quad
 (M_i) \longmapsto \boxtimes_{i \in I} M_i.
\]
%We use the same symbol $\boxtimes$ for left $\shD$-modules.
The functors $f_*$ and $f^!$ are compatible with $\boxtimes$.

Now we can discuss pseudo-tensor structures on $\catM(X)$.
Let $\catS$ be the category of finite non-empty sets 
and surjective maps.
For a morphism $\pi:J \twoheadrightarrow I$ and $i \in I$, 
we set $J_i := \pi^{-1}(i) \subset J$.
We will also denote $n := \{1,2 \ldots, n\} \in \catS$ 
if confusion may not occur.
For $I \in \catS$, we denote the diagonal embedding by 
\[
 \Delta^{(I)}: X \longhookrightarrow X^I.
\]
%For $L_i^{\bullet} \in \catDM^\ell(X)$, %and $L_i \in \catM^\ell(X)$, 
%we have 
%\[
% \stackrel{\bfL}{\otimes} L_i^\bullet 
%  \simeq \Delta^{(I) !}(\boxtimes L_i^\bullet). %, \quad
%% \otimes L_i \simeq 
%%  \Delta^{(I) *}(\boxtimes L_i).
%\]
%Here $\stackrel{\bfL}{\otimes}$ is the left derived functor 
%of the tensor product $\otimes$ on $\catM^\ell(X)$.

For $I \in \catS$ and $L_i,M \in \catM(X)$ with $i\in I$, set 
\[
 P^*_I(\{L_i\}_{i \in I},M) := 
 \Hom_{\catM(X^I)}\left(\boxtimes_{i \in I} L_i, \Delta^{(I)}_* M\right).
\]
Then for a surjection $\pi: J \twoheadrightarrow I$,
one can define a multi-linear map 
\begin{equation}\label{eq:pt:comp}
 \gamma^*_{\pi}:
 P^*_I\left(\{L_i\}_{i \in I},M\right) \otimes
 \bigotimes_{i \in I} P^*_{J_i}\left(\{K_j\}_{j \in J_i},L_i\right) 
 \longto 
 P^*_J\left(\{K_j\}_{j \in J},M\right) 
%,\qquad
%(\varphi,(\psi_i)_{i\in I}) \longmapsto 
%\gamma^*_{\pi}(\varphi,\psi_i) \equiv \varphi(\psi_i)
\end{equation}
by composing morphisms of $\shD$-modules.
Here $\otimes := \otimes_{\bbK}$ 
is the tensor product of $\bbK$-vector spaces.
Explicitly, $\gamma^*_{\pi}(\varphi,(\psi_i)_{i \in I})$ is given by
\begin{align*}
 \boxtimes_{j \in J} K_j = 
 \boxtimes_{i \in I}\boxtimes_{j \in J_i} K_j
 \xrightarrow{\ \boxtimes_{i \in I} \psi_i \ }
 \boxtimes_{i \in I} \Delta^{(J_i)}_* L_i = 
 \Delta^{(\pi)}_*\left(\boxtimes_{i \in I}L_i\right)
 \xrightarrow{\ \Delta^{(\pi)}(\varphi) \ }
 \Delta^{(\pi)}_*\Delta^{(I)}_* M
 = \Delta^{(J)}_* M,
\end{align*}
where we used
\[
 \Delta^{(\pi)} := \prod_{i \in I}\Delta^{(J_i)}: X^I \longhookrightarrow X^J
\]
and the natural identification 
$\Delta^{(\pi)} \Delta^{(I)} = \Delta^{(J)}$.
Let us call $\gamma^*_\pi$ the \emph{composition}.

The compositions are associative in the following sense.
Let us use a simplified symbol 
\[ 
 \varphi(\psi_i) := \gamma^*_{\pi}\left(\varphi,(\psi_i)_{i \in I}\right)
\]
for a composition of morphisms.
If $\rho: H \twoheadrightarrow J$ is another surjective map, 
$\{F_h\}_{h \in H}$ an $H$-family, 
and if $\chi_j \in P^*_{H_j}(\{F_h\}_{h \in H_j}, K_j)$,
then in $P^*_{H}(\{F_h\}_{h \in H},M)$ one has 
\begin{equation}\label{eq:pt:assoc}
 \varphi(\psi_i(\chi_j)) = (\varphi(\psi_i))(\chi_j) 
\end{equation} 

We also have 
$\eta_M := \id_M \in \Hom_{\catM(X)}(M,M) = P^*_{\{1\}}(\{M\},M)$.
Then for any $\varphi \in P^*_I(\{L_i\}_{i \in I}, M)$ 
we have
\begin{equation}\label{eq:pt:unit}
\eta_M(\varphi) = \varphi(\eta_{L_i}) = \varphi.
\end{equation}

Although we don't need the following definition in a full generality,
let us recall 

\begin{dfn*}[{\cite[\S1.1]{BD}}]
A pseudo-tensor category $(\catM,P,\gamma,\eta)$ 
consists of the following data.
\begin{itemize}
\item
A class $\catM$  of objects,
\item 
A set 
\[
 P^{\catM}_I\left(\{L_i\}_{i \in I},M\right) \equiv 
 P_I\left(\{L_i\},M\right)
\]
for any $I \in \catS$, any $I$-family of objects 
$\{L_i\}_{i \in I}$ in $\catM$, and any $M \in \catM$.
It is called the set of \emph{$I$-operations}.
%whose elements are called $I$-operations.

\item
A map 
\[
 \gamma_{\pi}:
 P_I\left(\{L_i\}_{i \in I},M\right) \times 
 \prod_{i \in I} P_{J_i}\left(\{K_j\}_{j \in J_i},M\right) 
 \longto 
 P_J\left(\{K_j\}_{j \in J},M\right) 
%,\quad
%(\varphi,(\psi_i)_{i\in I}) \longmapsto 
%\gamma_{\pi}(\varphi,(\psi_i)_{i \in I}) \equiv \varphi(\psi_i)
\]
for any morphism $\pi:J \to I$ in $\catS$,
any families $\{L_i\}_{i \in I}$ and $\{K_j\}_{j \in J}$,
and any $M \in \catM$.
It is called the \emph{composition map}.

\item
An element $\eta_M \in P_{{\{1\}}}(\{M\},M)$ for any $M \in \catM$ 
called the \emph{unit operation}. 
\end{itemize}
These should satisfy the following conditions.
\begin{enumerate}
\item 
The composition map is associative in the sense \eqref{eq:pt:assoc}.
\item
The equality \eqref{eq:pt:unit} holds 
for any $\varphi \in P_I(\{L_i\},M)$.
\end{enumerate}
\end{dfn*}

A pseudo-tensor category $\catM$ is naturally a category in the usual sense 
by setting the set of morphisms $\Hom_{\catM}(M,N)$ to be $P_1(\{M\},N)$ 
and the identity functor to be $\eta_M$.
%We say that $P=(P_I)$ is a pseudo-tensor structure on a category $\catM$ if 
%$\left(\catM,(P_I), (\eta_M)\right)$ is a pseudo-tensor category extending 
%the original category $\catM$.

\begin{rmk}\label{rmk:PTC:A}
One can modify the definition of pseudo-tensor category
by requiring $P_I$ to be an object of a fixed tensor category $\catA$.
Such a structure will be called a pseudo-tensor $\catA$-category.
If $\catA$ is the tensor category of $R$-modules, 
where $R$ is a commutative ring,
then it is called a pseudo-tensor $R$-category.
Hereafter we will always work on a pseudo-tensor $\bbK$-category,
and suppress the phrase `$\bbK$-'.
\end{rmk}

Using these notions, we can state 

\begin{dfn}\label{dfn:*PRS}
For any smooth scheme $X$ over $\bbK$,
we have a pseudo-tensor category 
\[
 \catM(X)^* := (\catM(X), P^*,\gamma^*,\eta)
\]
and call it the \emph{$*$-pseudo-tensor structure} on $\catM(X)$.
\end{dfn}

Note also that a tensor category $(\catM,\otimes)$  
has a pseudo-tensor structure by 
$P_I(\{L_i\},M) := \Hom_{\catM}(\otimes_{i \in I} L_i,M)$.
In particular, 
we have a pseudo-tensor category 
\[
 \catM^!(X) := (\catM(X),\otimes^!).
\] 

%\begin{rmk*}
%If fact, $\catM(X)^*$ is augmented
%with the augmentation functor given by the de Rham functor
%\[
% h(M) := M \otimes_{\shD_X} \shO_X = M/M \Theta_X.
%\]
%An \emph{augmented pseudo-tensor category} is defined by 
%replacing $\catS$ in the definition of pseudo-tensor category
%by $\wh{\catS}$, the category of arbitrary (so possibly empty) 
%finite sets and maps between them.
%By \cite[\S1.2.6.\ Lemma]{BD}, 
%an augmented pseudo-tensor category is equivalent to 
%a pseudo-tensor category $\catM$ 
%together with an \emph{augmentation functor}.
%The latter is a functor 
%$h: \catM \longto \catSet$ 
%and maps 
%$h_{I,i_0}: P_I(\{L_i\},M) \times h(L_{i_0})
% \longto P_{I \setminus \{i_0\}}(\{L_i\},M)$
%for any $I \in \S$ with $|I| \ge 2$ and $i_0 \in I$,
%satisfying some compatibility conditions 
%which we omit and refer \cite[\S1.2.5]{BD}.
%\end{rmk*}

A pseudo-tensor category having only one object 
is nothing but a reduced operad.
In other words,
for a pseudo-tensor category $\catM$, 
one can define an operad for each $M \in \catM$ by setting
\[
 \opEnd^{\catM}_M := \otimes_{n\ge1} P^{\catM}_n(M),
 \quad
 P^{\catM}_n(M):=P^{\catM}_{[n]}(\{M,\ldots,M\},M).
\]
%corresponds to the sets of $n$-array operations on $L$.
The $\frkS_n$-action on each factor is given by the compositions 
with respect to bijections from $[n]=\{1,\ldots,n\}$ to itself.
The compositions and $\eta_L$ yield an operad structure on 
the $\frkS$-module $\opEnd^{\catM}_L$.

\begin{dfn*}
For an operad $\opB$ and a pseudo-tensor category $\catM$,
a \emph{$\opB$-algebra in $\catM$} is an object $L \in \catM$ 
with an operad morphism $\opB \to \opEnd^{\catM}_L$.
Denote by $\opB(\catM)$ the category of $\opB$-algebras in $\catM$.
\end{dfn*}

For later use,
we define the following symbols for an operad $\opB$.
\[
 \opB^*(X) := \opB(\catM(X)^*), \quad 
 \opB^!(X) := \opB(\catM(X)^!).
\]

Let us close this subsection by recalling

\begin{dfn*}[{\cite[\S1.1]{BD}}]
A \emph{pseudo-tensor functor} $\catM \to \catN$ of 
two pseudo-tensor categories $\catM$ and $\catN$ is a functor 
\[ 
 \tau: \catM \longto \catN
\] 
of categories together with a map 
\[
 \tau_I: 
  P^{\catM}_I\left(\{L_i\}_{i \in I},M\right)
  \longto
  P^{\catN}_I\left(\{\tau(L_i)\}_{i \in I},\tau(M)\right) 
\]
for any $I \in \catS$
so that $\tau_I$ are compatible with compositions and 
$\tau_1(\id_M) = \id_{\tau(M)}$ for any $M \in \catM$.
\end{dfn*}

Thus a pseudo-tensor functor between operads 
(seen as pseudo-tensor categories with one object)
is nothing but a morphism of operads.

%%%%%%%%%%%%%%%%%%%%%%%%%%%%%%%%%%%%%%%%%%%%%%%%%%%%%%%%%%%%%%%%%%%%%%
%%%%%%%%%%%%%%%%%%%%%%%%%%%%%%%%%%%%%%%%%%%%%%%%%%%%%%%%%%%%%%%%%%%%%%
\subsection{Coisson algebras}
\label{subsec:cois}

We follow \cite[\S1.4]{BD} to introduce coisson algebras.
Our definition is due to {\cite[\S1.4.28.\ Lemma]{BD}} 
and  different from the original definition.

Recall that $\catS$ denotes the category of finite non-empty sets 
and surjections.
An operad $\opB$ 
%(whose underlying $\frkS$-module is defined over $\bbK$),
can be enlarged to a functor over $\catS$ by
\[
 I \longmapsto \opB_I:= 
 \left(\oplus_{f} \opB(n)\right)_{\frkS_{n}}.
\]
Here we set $n:=|I|$ and $f$ runs over the bijections 
$f:I \to \{1,\ldots,n\}$.
The last term means the module of $\frkS_{n}$-coinvariants.
We use this construction for $\opB=\opLie$ below.

Let $X$ be a smooth scheme over $\bbK$ as before.
We introduce a new  pseudo-tensor structure 
\begin{equation}\label{eq:MXc}
 \catM(X)^c=(\catM(X), P^c,\gamma^c,\eta)
\end{equation}
on $\catM(X)$ by setting 
\[
 P^c_I\left(\{L_i\}_{i \in I},M\right) := 
  \oplus_{S \in Q(I)} P^c_I\left(\{L_i\},M\right)_S,
 \quad
 P^c_I\left(\{L_i\}_{i \in I},M\right)_S := 
  P^*_S\left(\{\otimes^!_{i \in I_s}L_i\}_{s \in I},M\right) \otimes  
  (\otimes_{s \in S}\opLie_{I_s}).
\]
for $I \in \catS$.
Here $Q(I)$ is the set of all the surjections $I \twoheadrightarrow S$.
The composition $\gamma^c$ is given by the tensor product of $\gamma^*$ 
and $\opLie$ operations.

As before, one has an operad for each $M \in \catM(X)$ induced by $\catM(X)^c$.

\begin{dfn}
For $M \in \catM(X)$,
the reduced operad $\opEnd^c_M$ arising from  $\catM(X)^c$
is called \emph{the coisson operad} on $M$.
\end{dfn}

Thus the underlying $\frkS$-module of $\opEnd^c_M$ is given by 
$\opEnd^c_M(n) = P^c_n(M)$.

\begin{dfn}[{\cite[\S1.4.28.\ Lemma]{BD}}]
\label{dfn:coisson}
A coisson algebra (without unit) on $X$ is 
a $\opLie$-algebra in $\catM(X)^c$.
\end{dfn}

In other words, a coisson algebra structure on $M \in \catM(X)$ 
is an operad morphism $\opLie \to \opEnd^c_M$.

\begin{rmk*}
%\label{rmk:cts}
Let us say a few words on the original definition.
Our category $\catM(X)$ with the $*$-pseudo tensor structure $P^*$
and the tensor structure $\otimes^!$ 
has an compound pseudo-tensor structure \cite[\S1.3]{BD}.
A coisson algebra is originally defined for 
any abelian augmented compound tensor category $\catM^{*!}$.

A compound tensor category $\catM^{*!}=(\catM,P^*,\otimes^!)$ 
consists of a pseudo-tensor structure $\catM^*=(\catM,P^*)$ on 
a category $\catM$
and a (usual) tensor structure $\catM^{\circ !}=(\catM^\circ,\otimes^!)$
on the dual category $\catM^\circ$ 
satisfying some duality relation.
We also skip the explanation of abelian property and augmentation. 
Denote by $\catM^!$ the tensor category dual to $\catM^{\circ !}$. 

For an operad $\opB$, 
we call $\opB$-algebras in $\catM^*$ simply \emph{$\opB^*$-algebras},
and those in $\catM^!$ simply \emph{$\opB^!$-algebras}.
(Here we consider the tensor category $\catM^!$ 
 as a pseudo-tensor category.)

Let us denote by $\opComu$ 
the operad of commutative associative algebra with unit.
%Then a $\opComu^!$-algebra, 
%namely a pseudo-tensor functor $\opComu \to \catM^!$, 
%is equivalent to an object $A \in \catM$ together with 
%a commutative and associative operation $\cdot_A: A \otimes^! A \to A$ 
%and a morphism $1_A: \obu \to A$ which is a unit for $\cdot_A$. 
%$\opComu(\catM^!)$ denotes the category of $\opComu$-algebras in $\catM^!$.
%It is a tensor category by $\otimes^!$ and 
%$\obu$ is a unit of this tensor structure.
%, and by $\obu$ the unit in $\catM^{*!}$.
Finally, a \emph{coisson algebra} %(abbreviation of \emph{compound Poisson}) 
in $\catM^{*!}$ 
is a pair $(A,\{ \, \})$ of $\opComu^!$-algebra $A$ 
which is also a $\opLie^*$-algebra,
such that the Lie bracket $\{ \, \} \in P^*_2(\{A,A\},A)$ 
satisfies the Leibniz rule with respect to commutative${}^!$ product.
\end{rmk*}

%%%%%%%%%%%%%%%%%%%%%%%%%%%%%%%%%%%%%%%%%%%%%%%%%%%%%%%%%%%%%%%%%%%%%%
%%%%%%%%%%%%%%%%%%%%%%%%%%%%%%%%%%%%%%%%%%%%%%%%%%%%%%%%%%%%%%%%%%%%%%
\subsection{Chiral algebras}
\label{subsec:CA}

We follow \cite[\S3.1]{BD} to introduce chiral algebras.

Recall that $\catS$ denotes the category of finite non-empty sets 
and surjections.
For $I \in \catS$ and a smooth scheme $X$, 
denote by $U^{(I)}$ the complement of diagonal divisors. 
Namely
\[
 U^{(I)} := \{(x_i) \in X^I \mid 
  x_i \neq x_j \text{ for any } i \neq j\}.
\]
The open embedding is denoted by 
\[
 j^{(I)}: U^{(I)} \longhookrightarrow X^I.
\]

For $L_i, M \in \catM(X)$ ($i \in I$), we set 
\[
 P^{\ch}_I(\{L_i\},M) := 
 \Hom_{\catM(X^I)}\bigl(
  j^{(I)}_* j^{(I) *}(\boxtimes_{i\in I} L_i),\Delta^{(I)}_* M \bigr).
\]
Elements of this set are called chiral operations.

\begin{rmk*}
Here is the explanation of the functor $f^*$ used in the definition.
For a morphism $f:X \to Y$ of smooth schemes,
the standard derived functor $f^!:\catDM(Y) \to \catDM(X)$ 
is compatible with the standard pull-back functors 
$\bfR f^!, \bfL f^*:\catDM_{\shO}(Y) \to \catDM_{\shO}(X)$ for $\shO$-modules.
Thus $f^!$ is right exact with respect to the  $t$-structure $\catM^\ell(X)$. 
We denote the corresponding functor as 
$f^*: \catM^\ell(Y) \longto \catM^\ell(X)$.
Now by the standard equivalence, 
we read this functor as $f^*:\catM(Y) \longto \catM(X)$.
\end{rmk*}

\begin{fct*}[{\cite[\S3.1.2]{BD}}]
There is a pseudo-tensor structure $(\catM(X),P^{\ch},\gamma,\eta)$
on $\catM(X)$. 
%,
%which has an augmentation functor extending the de Rham functor $h$.
\end{fct*}

The resulting %augmented 
pseudo-tensor category is denoted by $\catM(X)^{\ch}$
and called the \emph{chiral pseudo-tensor structure}.
For an operad $\opB$, the category of $\opB$-algebras 
in $\catM(X)^{\ch}$ will be denoted by $\opB^{\ch}(X)$.

\begin{dfn}
The operad structure induced by $\catM(X)^{\ch}$ on 
\[
 \opEnd_M^{\ch}:= \oplus_{n \ge 1} \opEnd_M^{\ch}(n), \quad
 \opEnd_M^{\ch}(n) \equiv P^{\ch}_n(M) := P^{\ch}_{[n]}(\{M,\ldots,M\},M)
\]
with $[n]=\{1,\ldots,n\}$ 
is called the \emph{chiral operad}.
\end{dfn}

Roughly speaking, a chiral algebra is a $\opLie$-algebra 
in $\catM(X)^{\ch}$,
but we require it to have a unit.
In order to state the precise definition, we need a few more notions.

For $M \in \catM(X)$, the \emph{unit operation} 
$\ve_M \in P^{\ch}_{[2]}(\{\omega_X,M\},M)$ 
is defined to be the composition 
\[
 j_* j^* \omega_X \otimes M 
 \longto (j_* j^* \omega_X \boxtimes M)/\omega_X \boxtimes M
 \longsimto \Delta_* M,
\]
where the last arrow comes from the canonical isomorphism 
$\omega_X \otimes^! M \simto M$.

%For an operad $\opB$, 
%a \emph{$\opB^{\ch}$-algebra on $X$} means 
%a $\opB$-algebra in the pseudo-tensor category $\catM(X)^{\ch}$.
%We write the category of $\opB^{\ch}$-algebras  by
%$\opB^{\ch}(X) := \opB(\catM(X)^{\ch})$.

Consider %$A \in \opLie^{\ch}(X)$.
a $\opLie$-algebra $A$ in $\catM(X)^{\ch}$.
It has an operation 
$\mu_A \in P^{\ch}_{[2]}(\{A,A\},A)$
coming from the bracket in $\opLie$.
%$\mu_A$ is called the \emph{chiral product}.
A \emph{unit} in  $A$ is a morphism of $\shD$-modules 
$1_A: \omega_X \longto A$
such that 
$\mu_A(1_A,\id_A) \in P^{\ch}_2(\{\omega_X,A\},A)$ 
coincides with the unit operation $\ve_A$.

\begin{dfn}
A \emph{chiral algebra} on $X$ is a $\opLie^{\ch}$-algebra in $\catM(X)^{\ch}$
with unit.
Denote by $\CA(X)$ %\subset \Lie^{\ch}(X)$ 
the category of chiral algebras and morphisms preserving units.
\end{dfn}

%%%%%%%%%%%%%%%%%%%%%%%%%%%%%%%%%%%%%%%%%%%%%%%%%%%%%%%%%%%%%%%%%%%%%%
%%%%%%%%%%%%%%%%%%%%%%%%%%%%%%%%%%%%%%%%%%%%%%%%%%%%%%%%%%%%%%%%%%%%%%
\subsection{Classical limit}
\label{subsec:cl}

The next goal is to explain the limit construction of coisson algebras from 
chiral algebras.
For a preparation, we introduce the classical limit pseudo-tensor structure 
following \cite[\S\S 3.1--3.2]{BD}.

Hereafter we assume $\dim X =1$.

%%%%%%%%%%%%%%%%%%%%%%%%%%%%%%%%%%%%%%%%%%%%%%%%%%%%%%%%%%%%%%%%%%%%%%
\subsubsection{The Lie operad in terms of chiral operations}

Consider the chiral operad $\opEnd^{\ch}_{\omega_X}$ 
on $\omega_X$.
Set
\begin{equation}\label{eq:lambda_I}
 \lambda_I := (\bbK[1])^{\otimes I} [-|I|].
\end{equation}
The group $\Aut(I)$ acts on it by the sign character
(see Remark \ref{rmk:sign} for our treatment of tensor product of complexes).
If $|I|=2$, then the residue morphism 
\[
 \Res: j^{(I)}_* j^{(I) *}\omega_{X^I} \longto 
  \Delta^{(I)}_* \omega_X
\]
yields a map 
\[
 r_I: \lambda_I \longto P^{\ch}_I(\omega_X).
\]

\begin{fct}[{\cite[\S3.1.5.\ Theorem]{BD}}]
\label{fct:Cohen}
For $\dim X = 1$, 
there is a unique isomorphism of operads
\[
 \kappa: \opLie \longsimto \opEnd^{\ch}_{\omega_X}
\]
which coincides with $r_I$ for $|I|=2$.
\end{fct}

\begin{rmk*}
In \cite[\S3.1.11]{BD} this fact is called Cohen's theorem,
which originally states that 
the Gerstenhaber operad is isomorphic to the homology of 
the little cube operad of dimension $2$.
We refer \cite[\S3.1.11]{BD} for th reason of this naming 
and the relation of 
$P^{\ch}$ to the homology of configuration spaces.
\end{rmk*}

We will not repeat the proof of this fact,
but recall some notions for later use.
For a $\shD_{X^I}$-module $M$,
a  \emph{special filtration} on $M$ 
is a finite increasing filtration $W_{\bullet}$ 
such that every graded component 
$\gr^{W}_m M := W_m M/W_{m-1}M$ is a 
finite sum of copies of $\Delta^{(I/T)}_* \omega_{X^T}$
for some $T \in Q(I,-m)$ 
with
\[%\begin{equation}\label{eq:Q(I,m)}
 Q(I,m) := \{S \in Q(I) \mid |S|=m\}.
\]%\end{equation}
We set $W_0 M := M$, so that a special filtration $W_{\bullet}$ 
looks like
\[
 W_{-|I|}M \subset W_{-|I|+1}M \subset \cdots 
 \subset W_{-1}M \subset W_0M = M.
\]

An important step in the proof of the above fact is to 
show that $j^{(I)}_* j^{(I) *} \omega_{X^I}$ 
has a special filtration $W_{\bullet}$ \cite[\S3.1.7.\ Lemma]{BD}.
It comes from the Cousin complex of $\omega_{X^I}$, but 
we skip the argument and cite 
an explicit description of its graded component from 
\cite[\S3.1.10]{BD} for later use. %in \S\ref{sss:cl}.
For $T \in Q(I)$, set the vector space 
\[
 \opLie_{I/T} := \otimes_{t \in T} \, \opLie_{I_t}.
\]
Then there is a canonical isomorphism 
\begin{equation}\label{eq:spfilt:grd}
 \gr^{W}_{-m} j^{(I)}_*j^{(I)*} \omega_X^{\boxtimes I}
 \longsimto 
 \bigoplus_{T \in Q(I,m)} 
 \Delta^{(I/T)}_* \omega_X^{\otimes T} \otimes \opLie^*_{I/T}
\end{equation}
Here $\opLie^*_{I/T}$ is the linear dual of $\opLie_{I/T}$,
and $\Delta^{(I/T)}$ denotes the embedding associated to 
the surjection $I \twoheadrightarrow T$:
\[
 \Delta^{(I/T)}:= \prod_{t \in T} \Delta^{(I_t)}: 
 X^T \longhookrightarrow X^I.
\]

%%%%%%%%%%%%%%%%%%%%%%%%%%%%%%%%%%%%%%%%%%%%%%%%%%%%%%%%%%%%%%%%%%%%%%
\subsubsection{The classical limit pseudo-tensor structure}
\label{sss:cl}

Now we will introduce a new pseudo-tensor structure $\catM(X)^{\cl}$ 
and explain the following sequence of pseudo-tensor functors.
\[
 \catM(X)^{\ch} \longto \catM(X)^{\cl} \longto \catM(X)^{c}.
\]

For the introduction of $\catM(X)^{\cl}$, 
let us recall the special filtration $W_{\bullet}$ on 
$j^{(I)}_*j^{(I)*} \omega_X^{\boxtimes I}$.
Since 
\begin{equation}\label{eq:jj:isom}
 j^{(I)}_*j^{(I)*}\bigl(\boxtimes_{i\in I} L_i\bigr) \simeq 
 \bigl(j^{(I)}_*j^{(I)*} \omega_X^{\boxtimes I}\bigr) 
 \otimes \bigl(\boxtimes_{i \in I} L_i^\ell\bigr),
\end{equation}
the special filtration induces another finite filtration on 
$j^{(I)}_*j^{(I)*}(\boxtimes L_i)$,
which we denote by the same symbol $W_{\bullet}$.
It yields a canonical filtration on the space of chiral operations as 
\begin{equation}\label{eq:ch:filt}
 P^{\ch}_I(\{L_i\},M)^n := 
 \Hom_{\catM(X^I)}\bigl(
  j^{(I)}_*j^{(I)*} (\boxtimes L_i) \big/ 
  W_{n-1-|I|} \, j^{(I)}_*j^{(I)*} (\boxtimes L_i),
  \Delta^{(I)}_* M
 \bigr)
\end{equation}
Note that it is an decreasing filtration.
We denote by $\gr^n := P^{\ch,n}/P^{\ch,n+1}$ 
its $n$-th graded component.

This filtration is compatible with the composition of chiral operations,
hence we have 

\begin{dfn*}
The graded components 
\[
  P^{\cl}_I := \gr^{\bullet} P^{\ch}_I
\]
give a pseudo-tensor structure on $\catM(X)$ 
called the \emph{classical limit} of the chiral structure.
We denote it by $\catM(X)^{\cl}$.
\end{dfn*}

Next we explain that there is an embedding
$\catM(X)^{\cl} \hookrightarrow \catM(X)^c$.
The graded component \eqref{eq:spfilt:grd} of $W_{\bullet}$ 
and the isomorphism \eqref{eq:jj:isom} 
induce the canonical surjection
\[
 \phi_I: \bigoplus_{T \in Q(I,m)}\Delta^{(I/T)}_*
  \boxtimes_{t \in T}\left(\left(\otimes^!_{i \in I_t} L_i\right)
  \otimes_{\bbK} \opLie^*_{I_t}
 \right)
 \longtwoheadrightarrow
 \gr^{W}_{-m}   j^{(I)}_*j^{(I)*} (\boxtimes L_i). 
\]
It is induced by the canonical map 
$\Delta_*^{(I/T)}\left(
 \boxtimes_{t \in T}\left(\otimes_{i \in I_t}L_i^\ell\right)\right)
 \to \boxtimes_{i \in I}L_i^\ell$,
where $\otimes$ over $L_i^\ell$ is the tensor structure on $\catM^\ell(X)$.
So $\phi_I$ is actually an isomorphism if $L_i$ are $\shO_X$-flat.
$\phi_I$ induces the following canonical embedding
\begin{equation}\label{eq:Pcl-Pc}
 \gr^n P^{\ch}_{I}(\{L_i\},M) \longhookrightarrow 
  \bigoplus_{T \in Q(I,|I|-n)} 
  P^*_T\left(\{\otimes^!_{i \in I_t} L_i\}_{t\in T}, M\right)
  \otimes \opLie_{I/T}.
\end{equation}

Let us write down this embedding explicitly.
For simplicity, we denote 
\[
 L := j^{(I)}_*j^{(I)*} (\boxtimes L_i), \quad
 N := \boxtimes_{t \in T}\left(
 \left(\otimes^!_{i \in I_t} L_i\right) \otimes_{\bbK} \opLie^*_{I_t}\right).
\]
Given $\varphi \in P^{\ch}_{I}(\{L_i\},M)^n$,
we have a composition of morphisms 
\[
 \oplus_{T \in Q(I,m-1)}\Delta^{(I/T)}_* N 
 \longtwoheadrightarrow
 \gr^{W}_{-(m-1)} L = W_{-(m-1)}L / W_{-m}L
 \longhookrightarrow L/W_{-m}L 
 \xrightarrow{\ \varphi \ } \Delta^{(I)}_* M.
\]
It is an element of the direct sum over $T$ of the spaces 
\begin{align*}
&\Hom_{\catM(X^I)}\bigl(\Delta^{(I/T)}_* N, \Delta^{(I)}_* M\bigr)
 \simeq
 \Hom_{\catM(X^I)}\bigl(\Delta^{(I/T)}_* N, 
 \Delta^{(I/T)}_* \Delta^{(T)}_* M\bigr) 
 \simeq
 \Hom_{\catM(X^T)}\bigl(N, \Delta^{(T)}_* M\bigr) \\
&\simeq
 \Hom_{\catM(X^T)}\bigl(
  \boxtimes_{t \in T}\left(\otimes^!_{i \in I_t} L_i\right), 
  \Delta^{(T)}_* M\bigr) \otimes_{\bbK} 
 \left(\otimes_{t\in T} \opLie_{I_t}\right)
 = P^*_T(N,M) \otimes \opLie_{I/T}
\end{align*}
Here the first isomorphism comes from the equality 
$\Delta^{(I)} = \Delta^{(I/T)} \Delta^{(T)}$,
and the second one is by the exactness of the functor $\Delta^{(T)}_*$.
This construction is independent of the choice of $\varphi$,
and we have the desired embedding.

Now recalling the pseudo-tensor structure $\catM(X)^c$ 
in \eqref{eq:MXc},
we see that the embedding \eqref{eq:Pcl-Pc} gives 
a canonical fully faithful embedding of pseudo-tensor categories
\[
 \catM(X)^{\cl} \longhookrightarrow \catM(X)^c
\]
which extends the identify functor on $\catM(X)$.

%%%%%%%%%%%%%%%%%%%%%%%%%%%%%%%%%%%%%%%%%%%%%%%%%%%%%%%%%%%%%%%%%%%%%%
%%%%%%%%%%%%%%%%%%%%%%%%%%%%%%%%%%%%%%%%%%%%%%%%%%%%%%%%%%%%%%%%%%%%%%
\subsection{Coisson algebra as the classical limit of chiral algebra}
\label{subsec:limit}

Now we argue that a coisson algebra may be obtained 
as a classical limit of chiral algebras.
As explained in the introduction,
it is an analogy of 
the limit construction of vertex Poisson algebras from 
vertex algebras.

%%%%%%%%%%%%%%%%%%%%%%%%%%%%%%%%%%%%%%%%%%%%%%%%%%%%%%%%%%%%%%%%%%%%%%
\subsubsection{Chiral structure and the compound tensor structure}

Now we explain the following pseudo-tensor functors
\begin{equation}\label{eq:M!-Mch-M*}
 \catM(X)^! \otimes \opLie \xrightarrow{\ \alpha\ }
 \catM(X)^{\ch} \xrightarrow{\ \beta \ }
 \catM(X)^*.
\end{equation}

As for $\beta$, first note that the natural morphism 
$\boxtimes L_i \to j^{(I)}_* j^{(I) *}(\boxtimes L_i)$
yields a map 
\[
 \beta_I: P^{\ch}_I(\{L_i\},M) \longto P^*_I(\{L_i\},M),
\]
which is compatible with the composition.
So it further yields a pseudo-tensor functor
\[
 \beta: \catM(X)^{\ch} \longto \catM(X)^*
\]
extending the identity functor on $\catM(X)$.
This functor respects the augmentations.

As for $\alpha$, 
%Next, there is a canonical faithful pseudo-tensor functor
%\[
% \alpha: \catM(X)^! \otimes \opLie \longto \catM(X)^{\ch}
%\]
%extending the identity functor on $\catM(X)$.
recall that $\catM(X)^!$ is a tensor category 
and $\opLie$ is a pseudo-tensor category associated to the operad structure. 
The tensor product $\catM(X)^! \otimes \opLie$ is a pseudo-tensor category 
whose space of operations is given by 
\[
 P_I(\{L_i\},M) = 
 \Hom_{\catM(X)}(\otimes^!_{i \in I} L_i, M) \otimes_{\bbK} \opLie_I.
\]
The pseudo-tensor functor $\alpha$ is given by 
\[
 \alpha_I: 
  \Hom_{\catM(X)}(\otimes^!_{i \in I} L_i, M) \otimes_{\bbK} \opLie_I
  \longto 
  P^{\ch}_I(\{L_i\},M)
\]
which sends $\varphi \otimes \mu$ to the chiral operation
\[
 j^{(I)}_*j^{(I) *}(\boxtimes L_i)
 \simeq 
 (\boxtimes L_i^\ell) \otimes j^{(I)}_*j^{(I)*} \omega_X^{\boxtimes I}
 \xrightarrow{\ \id \otimes \kappa(\mu) \ }
 (\boxtimes L_i^\ell) \otimes \Delta^{(I)}_* \omega_X
 \simeq 
  \Delta^{(I)}_*(\otimes^!L_i)
 \xrightarrow{\ \Delta^{(I)}_*(\varphi) \ }
 \Delta^{(I)}_* M,
\]
where
$\kappa$ %in the first arrow 
is the isomorphism in Fact \ref{fct:Cohen}.
$\alpha_I$ is injective for any $I$, so that 
$\alpha$ is a faithful pseudo-tensor functor.

%%%%%%%%%%%%%%%%%%%%%%%%%%%%%%%%%%%%%%%%%%%%%%%%%%%%%%%%%%%%%%%%%%%%%%
\subsubsection{Commutative chiral algebras}

For an operad $\opB$,
consider $\opB$-algebras on the pseudo-tensor categories 
in the sequence \eqref{eq:M!-Mch-M*}.
Then we have the sequence of functors 
\[
 \opB\left(\catM(X)^! \otimes \opLie\right) \xrightarrow{\ \alpha^{\opB} \ }
 \opB^{\ch}(X) \xrightarrow{\ \beta^{\opB} \ } \opB^*(X).
\]
Applying this argument to $\opB=\opLie$,
we have
\[
 \opCom^!(X)    \xrightarrow{\ \alpha^{\opLie} \ }
 \opLie^{\ch}(X) \xrightarrow{\ \beta^{ \opLie} \ }
 \opLie^*(X).
\]
Here $\opCom$ is the operad of commutative algebras (without unit), 
and we used the following fact.

\begin{fct*}[{\cite[1.1.10.\ Lemma]{BD}}]
For any pseudo-tensor category $\catM$,
$\opLie\left(\catM \otimes \opLie \right) \simeq \opCom(\catM)$.
\end{fct*}

Recall that a chiral algebra (or a $\opLie^{\ch}$-algebra in general) 
is equipped with the binary chiral operation %chiral bracket 
$\mu_A \in P^{\ch}_2(\{A,A\},A)$.

\begin{dfn*}
A $\opLie^{\ch}$-algebra $A$ is said to be commutative 
if $[\ ]_A =0$,
where 
\[
 [\ ]_A := \beta^{\opLie}(\mu_A) \in P^*_2(A) \:= P^*_{[2]}(\{A,A\},A).
\]
Denote by $\opLie^{\ch}_{com}(X) \subset \opLie^{\ch}(X)$ 
the full subcategory of commutative $\opLie^{\ch}$-algebras,
and by $\CA(X)_{com} \subset \CA(X)$ 
the full subcategory of commutative chiral algebras.
We call $[\ ]_A$ the \emph{$*$-bracket} of $A$.
\end{dfn*}

Finally we remark the equivalences
\[
 \alpha^{\opLie}: \opCom^!(X) \longsimto \opLie^{\ch}_{com}(X), \quad  
 \opComu^!(X) \longsimto \CA(X)_{com},
\]
where $\opComu$ is the operad of commutative algebra with unit.
One can obtain these from the the sequence \eqref{eq:M!-Mch-M*}
and the observation that  the composition 
$\beta_I \alpha_I: P^!_I \otimes \opLie_I \to P^*_I$ 
vanishes for $|I| \ge 2$ and the sequence 
$0 \to  P^!_2 \otimes \opLie_2 \to P^{\ch}_2 \to P^*_2$ is exact.

%%%%%%%%%%%%%%%%%%%%%%%%%%%%%%%%%%%%%%%%%%%%%%%%%%%%%%%%%%%%%%%%%%%%%%
\subsubsection{The deformation problem}

A coisson algebra %in the compound tensor category $\catM(X)^{*!}$ 
can be considered as classical limits of chiral algebras on $X$,
as indicated in \cite[\S3.3.11]{BD}.

Let $A_t$ be a flat family  of chiral algebras over $\bbK[t]$, 
namely it is a chiral algebra defined over $\bbK[t]$ 
which is flat as a $\bbK[t]$-module.
Assume that $A_{0} := A_t/t A_t$ is a commutative chiral algebra,
namely, the $*$-bracket $[ \, ]_{t}$ of $A_t$ 
vanishes modulo $t$.

\begin{fct*}
Under the assumption, $A_0$ has a structure of coisson algebra. 
\end{fct*}

In fact, $\{ \, \}_t := t^{-1}[\, ]_{t}$ is a $\opLie^*$-bracket on $A_t$,
and the corresponding $\opLie^*$-algebra acts on the chiral algebra $A_t$ 
by adjoint.
Modulo $t$, we see that $A_0 \in \opComu^!(X)$ 
and $\{ \, \} := \{ \, \}_t \pmod t$ is a coisson bracket.

We call $A_t$ a \emph{chiral deformation quantization} of the coisson algebra $(A_0,\{\, \})$ (although \cite{BD} simply called it quantization). 

Using the sequence 
$\catM^{\ch} \to \catM^{\cl} \hookrightarrow \catM^c$ 
of \eqref{eq:M!-Mch-M*}
and Definition \ref{dfn:coisson} that a coisson algebra is 
equivalent to a $\opLie$-algebra in $\catM^c$,
one can restate our deformation problem as follows.
For an object $A \in \catM(X)$, 
we have a series of operad morphisms 
\begin{equation}\label{eq:operad_seq} 
 \opEnd^{\ch}_A \longto \opEnd^{\cl}_{A} \longto \opEnd^{c}_{A}
\end{equation}
between the operads 
induced by the pseudo-tensor structures 
$\catM(X)^{\ch}$, $\catM(X)^{\cl}$ and $\catM(X)^{c}$.
Then a chiral and a coisson algebra structure on $A$ are given by 
operad morphisms
\[
 \opLie \longto \opEnd^{\ch}_A, \quad
 \opLie \longto \opEnd^{c}_A
\]
respectively.
Then a deformation of a coisson algebra $A$
is to lift the operad morphism from $\opEnd^{c}_A$ to $\opEnd^{\ch}_A$.
\[
\xymatrix{
& \opLie \ar[rd] \ar@{..>}[ld]
\\ 
\opEnd^{\ch}_A \ar[r] & 
\opEnd^{\cl}_A %\ar@{^{(}->}[r] 
\ar[r]  & \opEnd^{c}_A
}
\]

%%%%%%%%%%%%%%%%%%%%%%%%%%%%%%%%%%%%%%%%%%%%%%%%%%%%%%%%%%%%%%%%%%%%%%
%%%%%%%%%%%%%%%%%%%%%%%%%%%%%%%%%%%%%%%%%%%%%%%%%%%%%%%%%%%%%%%%%%%%%%
%%%%%%%%%%%%%%%%%%%%%%%%%%%%%%%%%%%%%%%%%%%%%%%%%%%%%%%%%%%%%%%%%%%%%%
\section{Convolution dg Lie algebra in operad formalism}
\label{sect:conv}

In \S \ref{subsubsec:convol-dgla} we give a brief explanation
of the convolution dg Lie algebra 
for $\opP$-algebra structures.
In this section we give a more detailed discussion 
following \cite{LV}.
%along the coisson or chiral setting.

%%%%%%%%%%%%%%%%%%%%%%%%%%%%%%%%%%%%%%%%%%%%%%%%%%%%%%%%%%%%%%%%%%%%%%
%%%%%%%%%%%%%%%%%%%%%%%%%%%%%%%%%%%%%%%%%%%%%%%%%%%%%%%%%%%%%%%%%%%%%%
\subsection{Convolution dg Lie algebra}
\label{subsec:conv-dgla}

%%%%%%%%%%%%%%%%%%%%%%%%%%%%%%%%%%%%%%%%%%%%%%%%%%%%%%%%%%%%%%%%%%%%%%
\subsubsection{Convolution pseudo-tensor category}

We will introduce a convolution pseudo-tensor category,
mimicking the discussion in \cite[\S6.4]{LV} for operads.

One may define the notion of \emph{co-pseudo-tensor category} in a dual way.
It is the data $(\catL,C,\Delta,\ve)$ consisting of the followings.
\begin{itemize}
\item
A class $\catL$ of objects.
\item
A $\bbK$-vector space of 
\emph{cooperations} 
$C_I^{\catL}\left(L,\{M_i\}_{i \in I}\right) 
 \equiv C_I\left(L,\{M_i\}\right)$
for any $I \in \catS$ and objects $L,M_i \in \catL$.
\item
A \emph{decomposition map} 
\[
 \Delta: 
 C_J(L,\{N_j\}_j) \longto 
 \bigoplus_{I \in Q(J)} 
 \bigoplus_{\{M_i\}_{i \in I}} C_I(L,\{M_i\}) \otimes_{\bbK} 
 \bigotimes_{\bbK, i \in I} C_{J_i}\left(M_i,\{N_j\}_{j \in J_i}\right)
\]
for any $J \in \catS$,
where the second summation is over all the $I$-families of objects $\{M_i\}$ 
in $\catL$.
\item
An element $\ve_N \in C_{\{1\}}(N,\{N\})$ for any $N \in \catL$ 
called the \emph{counit cooperation}.
\end{itemize}
These should satisfy  coassociativity and counit axiom 
like in the pseudo-tensor structure.

Similarly as in the operad case,
a co-pseudo-tensor category with one object coincides with a reduced  cooperad.
Also a co-pseudo-tensor category has a structure of the category 
in the usual sense.

Assume that we are given  a pseudo-tensor category 
$\catM=(\catM,P,\gamma,\eta)$ and co-pseudo-tensor category 
$\catL=(\catL,C,\Delta,\ve)$.
Fix $I \in \catS$ and consider the $\bbK$-module
\[
 P^{\Hom}_I \left(K,\{L_i\};\{M_i\},N\right)
 :=\Hom_{\bbK}\left(C_I\left(K,\{L_i\}\right),P_I\left(\{M_i\},N\right)\right).
\]
for $K,L_i \in \catL$ and $M_i,N \in \catM$ ($i \in I$).
We claim that there is a pseudo-tensor category structure
over the product category $\catL \otimes_{\bbK} \catM$ (in the usual sense)
encoded by the above $P^{\Hom}$.
Suppressing the symbols of objects such as $K$ and $L_i$'s, 
the composition map is written as 
\[
 \gamma^{\Hom}: 
 P^{\Hom}_I \otimes \bigotimes_{i \in I} P^{\Hom}_{J_i}
 \longto 
 P^{\Hom}_J
\]
for $I \in Q(J)$.
We define it  by sending $\varphi \otimes \left(\otimes_i \psi_i\right)$ to 
\begin{align*}
C_J \xrightarrow{\ \Delta\ \ } 
 \bigoplus_{I' \in Q(J)} C_{I'} \otimes \bigotimes_{i \in I'} C_{J_i}
 \longtwoheadrightarrow
C_I \otimes \bigotimes_{i \in I} C_{J_i} 
 \xrightarrow{\ \varphi \otimes \left(\otimes_i \psi_i\right) \ }
P_I \otimes \bigotimes_{i \in I} P_{J_i} 
 \xrightarrow{\ \gamma \ } P_J. 
\end{align*}
One can easily find the unit operation $\eta^{\Hom}$ using
$\eta$ and $\ve$.
We denote the resulting pseudo-tensor category as 
\[
 \Hom_{PT}(\catL,\catM) := 
 \left(\catL \otimes_{\bbK} \catM, P^{\Hom}, \gamma^{\Hom}, \eta^{\Hom}\right).
\]
We will call it \emph{the convolution pseudo-tensor category}.

Now let us introduce the dg version.

\begin{dfn*}
A \emph{dg pseudo-tensor category} is 
a pseudo-tensor $\catA$-category (see Remark \ref{rmk:PTC:A})
where $\catA$ is taken to be the abelian category of 
%chain 
complexes ($\bbZ$-graded vector spaces 
with endomorphisms $d$ of degree $1$ with $d^2=0$) over $\bbK$.
\end{dfn*}

Thus $P_I$ in a dg pseudo-tensor category $\catM=(\catM,P,\gamma,\eta)$ 
is a complex over $\bbK$.
A \emph{dg co-pseudo-tensor category} is similarly defined.

\begin{rmk}
\label{rmk:sign}
%Later we will consider tensor products of complexes.
We understand complexes are sign-graded,
namely the tensor product $V \otimes W$ is equipped with 
the symmetry
\[
 \tau: V \otimes W \longto W \otimes W, \quad
 v \otimes w  \longmapsto (-1)^{|v| \, |w|} w \otimes v.
\]
\end{rmk}

Let $\catM$ be a dg pseudo-tensor category and 
$\catL$ be a dg co-pseudo-tensor category.
Then the space $P^{\Hom}_I$ has a grading induced by 
the grading structures in $\catM$ and $\catL$.

\begin{dfn}\label{dfn:dgHom}
For a homogeneous element 
$f \in P^{\Hom}_I=\Hom_{\bbK}(C_I,P_I)$, 
define 
\[
 \partial(f) := d_P f - (-1)^{|f|}f d_C,
\]
where $|f|$ is the grading of $f$ and 
$d_P$, $d_C$ are the differentials of the complex 
$P_I$, $C_I$ respectively.
\end{dfn}

Then we immediately have $\partial^2=0$.
Thus we have 

\begin{lem*}%\label{lem:conv:dg}
$(\Hom_{PT}(\catL,\catM),\partial)$ is a dg pseudo-tensor category.
\end{lem*}

%%%%%%%%%%%%%%%%%%%%%%%%%%%%%%%%%%%%%%%%%%%%%%%%%%%%%%%%%%%%%%%%%%%%%%
\subsubsection{Infinitesimal composition maps}

The next goal is to introduce a dg Lie algebra associated to 
each object of the convolution pseudo-tensor category. 
As a preliminary, we introduce infinitesimal composition maps 
for operads and cooperads, following \cite[\S6.1]{LV}.
We also introduce several basic notations of (co)operads 
for later purpose.

%For any operad $P=(\oplus_{n\ge 1}P(n),\gamma,\eta)$ 
%we can associate a bracket operation
%\[
% \{f,g\}  := \sum_{i=1}^m \sum_{\lambda} 
% \left(f \underset{i}{\circ} g \right)^{\sigma_\lambda}
%\]
%for $f \in P(m)$ and $g \in P(n)$.
%Here we denoted the $i$-th partial operation as 
%\[
% f \underset{i}{\circ} g  := 
% \gamma(f;\eta,\ldots,\eta,g,\eta,\ldots,\eta) \in P(m+n-1),
%\]
%and $\lambda$ runs over the ordered partitions of 
%$(1,\ldots,1,\underbrace{n-i+1}_{\text{$i$-th}},1,\ldots,1)$.
%This bracket satisfies the pre-Lie condition,
%namely 
%\[
% \{\{ f,g\},h\} -\{f,\{g,h\}\}=\{\{f,h\},g\}-\{f,\{h,g\}\}
%\]
%Thus the anti-symmetrization 
%\[
% [f,g]: = \{f,g\} - \{g,f\}
%\]
%satisfies the Jacobi rule, so we have a Lie algebra.

First let us recall the composite of $\frkS$-modules.
Given $\frkS$-modules $M=\oplus_{n\ge0} M(n)$ and $N=\oplus_{n\ge0} N(n)$,
we set
\[
 M \circ N := 
  \bigoplus_{n\ge0} M(n) \otimes_{\frkS_n} N^{\otimes n},
\]
where $\frkS_n$ acts on $M(n)$ by the given right action 
and on $N^{\otimes n}$ by permutation of factors.
%Its $n$-ary part is given by 
%\begin{equation}
%\label{eq:MoN(n)}
%(M \circ N)(n) \simeq
% \bigoplus_{k\ge0} M(k) \otimes_{\frkS_k} 
% \Bigl(\bigoplus \Ind^{\frkS_n}_{\frkS_{i_1}\times \cdots \frkS_{i_n}}
% N(i_1) \otimes \cdots \otimes N(i_k)\Bigr),
%\end{equation}
%where $i_j$'s should satisfy $i_1+\cdots+i_k=n$.
Let us denote an element of 
$(M \circ N)(n) = M(n) \otimes_{\frkS_n} N^{\otimes n}$ by 
\[
 (\mu;\nu_1,\ldots,\nu_n) 
\]
with $\mu \in M(n)$ and $\nu_i \in N$.

Using the composite, 
we can consider 
the composition map $\gamma$ of an operad $\opP=(\opP,\gamma,\eta)$ 
as an $\frkS$-module morphism $\gamma: \opP \circ \opP \longto \opP$.
%We will denote the image of the composite as 
%\begin{equation}\label{eq:comp}
% \mu(\nu_1,\ldots,\nu_n) := \gamma(\mu;\nu_1,\ldots,\nu_n)
%\end{equation}
%with $\mu \in \opP(n)$ and $\nu_i \in \opP$.
Similarly the decomposition map $\Delta$ of a cooperad $\opC=(\opC,\Delta,\ve)$ is a morphism $\opC \to \opC \circ \opC$.
The image $\Delta(\mu)$ is written as 
\begin{equation}\label{eq:decomp}
 \Delta(\mu) = \sum (\nu;\nu_1,\ldots,\nu_n),\quad
 \nu \in \opC(n), \  \nu_i \in \opC. 
\end{equation}

For two morphisms $f:M_1 \to M_2$ and $g:N_1 \to N_2$ of $\frkS$-modules,
we have a natural composite
\[
 f \circ g : M_1 \circ N_1 \longto M_2 \circ N_2.
\]
It can be written down as 
\[
 (f \circ g)(\mu;\nu_1,\ldots,\nu_n) :=
 \left(f(\mu);g(\nu_1),\ldots,g(\nu_n)\right).
\]

For $\frkS$-modules $L$, $M$ and $N$, we set
\[
 L \circ(M;N) := 
  \bigoplus_{n\ge0} L(n) \otimes_{\frkS_n} \Bigl(\bigoplus_{i=1}^n
  M^{\otimes (i-1)} \otimes N \otimes M^{\otimes (n-i)}\Bigr), 
\]
where $N$ sits in the $i$-th position,
and the $\frkS_n$-actions are similar as in the composite $M \circ N$. 
Clearly it is a sub $\frkS$-module of the composite $L \circ (M \oplus N)$.
We also set
\[
 M \ccone N := M \circ (I;N),
\]
where $I$ denotes the identity  $\frkS$-module given by
\begin{equation}\label{eq:op:I}
 I = \oplus_{n \ge 0}I(n), \quad
 I(0)=0,\ I(1) = \bbK,\ I(2)=I(3)=\cdots=0.
\end{equation}
Thus an element of 
$(M \circ_{(1)} N)(n) = 
 M(n) \otimes_{\frkS_n} (\oplus_{i=1}^n
  I^{\otimes (i-1)} \otimes  N \otimes I^{\otimes (n-i)})$
can be written as 
\[
 (\mu;k_1,\ldots,k_{i-1},\nu,k_{i+1},\ldots,k_n)
\]
with $\mu \in M(n)$, $\nu \in N$ and $k_j \in I \simeq \bbK$. 
Now we have

\begin{dfn}\label{dfn:inf-comp:op}
For an operad $\opP=(\opP,\gamma,\eta)$,
we define the \emph{infinitesimal composition map} $\gamma_{(1)}$ to be
\[
\gamma_{(1)}: 
 \opP \ccone \opP \longhookrightarrow \opP \circ (I \oplus \opP) 
 \xrightarrow{\ \id_{\opP} \circ (\eta+\id_{\opP})\ }
 \opP \circ \opP \xrightarrow{\ \gamma\ } \opP.
\] 
\end{dfn}

By definition the map $\gamma_{(1)}$ 
is the restriction of $\gamma$ where we only compose two operations. 
%Namely, setting $\id:=\eta(1)\in\opP(1)$  and using \eqref{eq:comp} we have 
%\begin{equation}\label{eq:gamma(1)}
% \gamma_{(1)}(\mu;1,\ldots,1,\nu,1,\ldots,1)
% = \gamma\mu(\id,\ldots,\id,\nu,\id,\ldots,\id).
%\end{equation}
%Then we have 

\begin{dfn}\label{dfn:inf-comp:mor}
For two morphisms $f:M_1 \to M_2$ and $g:N_1 \to N_2$ 
of $\frkS$-modules, we define the morphism 
\[
 f \ccone g: M_1 \ccone N_1 \longto M_2 \ccone N_2,\quad
 (\mu;k_1,\ldots,\nu,\ldots,k_n) \longmapsto
 (f(\mu);k_1,\ldots,g(\nu),\ldots,k_n).
\]
%Here we denoted an element of $P_i \ccone Q_i$ by 
%$(\mu;k_1,\ldots,\nu,\ldots,k_n)$ 
%with $\mu \in P_i$, $\nu \in Q_i$ and $ k_j \in I$.
\end{dfn}

Next we recall the infinitesimal composition map of cooperad.
Let us introduce another composite $\circ'$
of two $\frkS$-module morphisms $f:M_1 \to M_2$ and $g:N_1 \to N_2$.
It is given by 
\[
 f \circ' g := \sum_{i}
 f \otimes \bigl(
 \id_{N_1}^{\otimes (i-1)} \otimes g \otimes \id_{N_1}^{\otimes (n-i)} 
 \bigr)
 : \  M_1 \circ N_1 \longto M_2 \circ (N_1;N_2),
\]
where $g$ sits in the $i$-th position. 
Then we have 

\begin{dfn}\label{dfn:inf-comp:coop}
For a cooperad $\opC=(\opC,\Delta,\ve)$,
the \emph{infinitesimal decomposition map} $\Delta_{(1)}$ 
of $\opC$ is defined to be
\[
 \Delta_{(1)}: 
 \opC \xrightarrow{\ \Delta \ } \opC \circ \opC
 \xrightarrow{\ \id_{\opC} \circ' \id_{\opC} \ }
 \opC \circ (\opC;\opC)
 \xrightarrow{\ \id_{\opC} \circ  (\ve;\id_{\opC}) \ }
 \opC \circ (I;\opC) = \opC \ccone \opC.
\] 
\end{dfn}

It is a decomposition of an element of $\opC$ into two parts.
Namely, using \eqref{eq:decomp} we have
\[
 \Delta_{(1)}(\mu) = 
 \sum \sum_{i=1}^n \left(\nu;\ve(\nu_1),\ldots,
  \ve(\nu_{i-1}),\nu_i,\ve(\nu_{i+1}),\ldots,\ve(\nu_n)\right)
\]

%%%%%%%%%%%%%%%%%%%%%%%%%%%%%%%%%%%%%%%%%%%%%%%%%%%%%%%%%%%%%%%%%%%%%%
\subsubsection{Convolution Lie algebra}

Now we will now define a Lie bracket on 
the convolution pseudo-tensor category $\Hom_{PT}(\catL,\catM)$
following \cite[\S6.4]{LV}

Fix an object $(L,M) \in \catL \otimes_{\bbK} \catM$.
Then we have a reduced operad 
$\opEnd_{M}=(\opEnd_{M},\gamma_M,\eta_M)$
with the underlying $\frkS$-module 
\[
 \opEnd_M = \oplus_{n\ge 1} \opEnd_M(n),\quad
 \opEnd_M(n) := P^{\catM}_n(M)=P^{\catM}_n(\{M,\ldots,M\},M).
\]
The composition map $\gamma_M$ comes from $\gamma$ of $\catM$.
Similarly we have a reduced cooperad
$\opcoEnd_{L} = (\opcoEnd_{L},\Delta_L,\ve_L)$ with
\[
 \opcoEnd_L = \oplus_{n\ge 1} \opcoEnd_L(n),\quad
 \opcoEnd_L(n) := C^{\catL}_n(L)=C^{\catL}_n(L,\{L,\ldots,L\}).
\]
We also have a reduced operad
$\opHom_{L,M} = (\opHom_{L,M},\gamma_{L,M},\eta_{L,M})$ with  
\[
 \opHom_{L,M}     = \oplus_{n\ge1}\opHom_{L,M}(n),\quad
 \opHom_{L,M}(n) %\equiv  P^{\Hom}_n\left(L;M\right)
               := \Hom_{\bbK}\left(\opcoEnd_L(n),\opEnd_M(n)\right).
\]
Below we simply denote 
\[
 \opP := \opEnd_M,\quad
 \opC := \opcoEnd_L,\quad
 \opHom \equiv \Hom_{\bbK}(\opC,\opP) := \opHom_{L,M}.
\]
Using Definitions \ref{dfn:inf-comp:op}--\ref{dfn:inf-comp:coop} 
we introduce

\begin{dfn*}
For $f,g \in \opHom$, define $f \star g$
to be the following composition of maps.
\[
 f \star g := \
 \bigl(\opC \xrightarrow{\ \Delta_{(1)} \ } 
 \opC \ccone \opC \xrightarrow{\ f \ccone g\ }
 \opP \ccone \opP \xrightarrow{\ \gamma_{(1)} \ } \opP\bigr).
\]
%Here we denoted $\opC := \opcoEnd_L$ and $\opP := \opEnd_M$.
\end{dfn*}

\begin{rmk*}
Precisely speaking, this definition is due to 
\cite[Proposition 6.4.3]{LV} and 
not the original definition in \cite[\S5.4.3]{LV} 
where partial composition is used.
\end{rmk*}

Using \eqref{eq:decomp} for $\mu \in \opC$ 
and denoting $\id = \eta(1) \in \opP(1)$,  
one can write down the definition as 
\begin{equation}\label{eq:star}
 (f \star g)(\mu)
= \sum \sum_{i=1}^n
 \ve(\nu_1)\cdots\ve(\nu_{i-1})\ve(\nu_{i+1})\cdots\ve(\nu_n)
 \cdot \gamma\left(f(\nu);\id,\ldots,\id,g(\nu_i),\id,\ldots,\id\right).
\end{equation}

\begin{prop}[{\cite[Proposition 6.4.3]{LV}}]
\label{prop:preLie}
The product $\star$ is a pre-Lie product, namely it satisfies
\[
 (f \star g) \star h - f \star (g \star h) = 
 (f \star h) \star g - f \star (h \star g).
\]
Hence the anti-symmetrization 
\[
 [f,g]: = f \star g - g \star f %\{f,g\} - \{g,f\}
\]
satisfies the Jacobi rule, 
and  $\opHom$ has a structure of Lie algebra.
\end{prop}

\begin{proof}
Although a proof is given in \cite[\S6.4]{LV},
let us give another proof using the explicit formula \eqref{eq:star}.
The coassociativity of $\Delta$ in $\opC$ is given by 
\[
 (\Delta \circ \id_{\opC}) \Delta = (\id_{\opC} \circ \Delta) \Delta: \ 
 \opC \longto 
 (\opC \circ \opC) \circ \opC \simeq \opC \circ (\opC \circ \opC),
\]
and on an element $\mu \in \opC$ with \eqref{eq:decomp} it means
\begin{align*}
&\sum 
 \bigl((\omega;\omega_1,\ldots,\omega_j,\ldots,);
  \nu_1,\ldots,\nu_i,\ldots,\nu_n\bigr)
\\
&=\sum  
 \bigl(\nu;
  (\omega_1;\omega_{1,1},\ldots),\ldots,
  (\omega_i;\omega_{i,1},\ldots,\omega_{i,j},\ldots),\ldots,
  (\omega_n;\omega_{n,1},\ldots)\bigr)
\end{align*}
with $\Delta(\nu)=\sum (\omega;\omega_1,\ldots)$ 
and  $\Delta(\nu_i)=\sum (\omega_i;\omega_{i,1},\ldots)$.
The formula \eqref{eq:star} yields
\begin{align*}
((f\star g) \star h)(\mu) =  \sum   
   \ve(\wh{\omega}_j) \ve(\wh{\nu}_i) \cdot 
   \gamma\bigl(\gamma(f(\omega);\id,\ldots, g(\omega_j),\ldots,\id);
   \id,\ldots, h(\nu_i),\ldots,\id\bigr)
\end{align*}
with $\ve(\wh{\omega}_j) := \prod_{k \neq j} \ve(\omega_k)$ 
and $\ve(\wh{\nu}_i) := \prod_{k \neq i} \ve(\nu_k)$.
Similarly we have
\begin{align*}
(f\star (g \star h))(\mu) = 
 \sum   
   \ve(\wh{\nu}_i) \ve(\wh{\omega}_{i,j}) \cdot 
   \gamma\bigl(f(\nu);
    \id,\ldots,
    \gamma(g(\omega_i);\id,\ldots,h(\omega_{i,j}),\ldots,\id),
    \ldots,\id\bigr)
\end{align*}
with $\wh{\omega}_{i,j}:= \prod_{k \neq j} \ve(\omega_{i,k})$. 
%The counit property \eqref{eq:decomp:nu_i} implies 
We also note that the counit property of $\ve$ in $\opC$ is given by
\[
 (\id_{\opC} \circ \ve) \Delta = (\ve \circ \id_{\opC}) \Delta = \id_{\opC}:\ 
 \opC \longto \opC \circ I \simeq I \circ \opC \simeq \opC.
\]
On the element $\nu_i$  it means
$\sum \prod_k \ve(\omega_{i,k}) \cdot (\omega_i;1,\ldots,1) 
 = \nu_i$,
where $(\omega_i;1,\ldots,1) \in \opC \circ I \simeq \opC$.
In particular,  we have
$\sum \prod_k \ve(\omega_{i,k}) g(\omega_i) = g(\nu_i)$
for $g \in \opHom_{L,M}$.
This formula and the associativity of $\gamma$ implies 
\begin{align*}
((f\star g) \star h)(\mu) =  \sum   
   \ve(\wh{\omega}_j) \ve(\wh{\nu}_i) \cdot 
   \gamma\bigl(
     \gamma(f(\omega);\id,\ldots, g(\omega_j),\ldots,\id);
     \id,\ldots, h(\nu_i),\ldots,\id
   \bigr)
\end{align*}
\begin{align*}
\bigl((f\star g) \star h - f\star (g \star h)\bigr)(\mu) = 
\sum_{i \neq j} 
   \prod_{k \neq i,j}\ve(\nu_k) \cdot 
   \gamma\bigl(f(\nu);
    \id,\ldots,g(\nu_i),\ldots, h(\nu_j),\ldots,\id\bigr).
\end{align*}
This expression is symmetric for $g$ and $h$, so that we have the result.
\end{proof}

\begin{dfn*}
Denote the  subspace of $\frkS$-equivariant morphisms in $\opHom$ by 
\[
 \Hom_{\frkS}(\opC,\opP)
  := \bigoplus_n \Hom_{\frkS_n}\bigl(\opC(n),\opP(n)\bigr).
\]
%For $(L,M) \in \catL \otimes_{\bbK} \catM$,
%let us denote
%\[
% \opHom^{\frkS}_{L,M} := \Hom_{\frkS}(\opcoEnd_L,\opEnd_M).
%\]
\end{dfn*}

The proof of Proposition \ref{prop:preLie} implies

\begin{lem*}[{\cite[Lemma 6.4.4]{LV}}]
$\Hom_{\frkS}(\opC,\opP)$ is stable under the pre-Lie product $\star$,
so that it is a Lie algebra.
\end{lem*}

We also see that 
the differential $\partial$ in  $\Hom_{PT}(\catL,\catM)$ 
in Definition \ref{dfn:dgHom} is compatible with this Lie bracket.
Hence 

\begin{fct*}[{\cite[Proposition 6.4.5]{LV}}]
For an operad $\opP$ and a cooperad $\opC$
we have a dg Lie algebra 
\[
 %\opHom^{\frkS}_{L,M} = 
 \bigl(\Hom_{\frkS}(\opC,\opP), [\, ], \partial\bigr).
\]
\end{fct*}

In particular, going back to our original situation, we have

\begin{cor}
\label{cor:g_CM}
Let $\catM$ be a dg pseudo-category, $M \in \catM$ and 
$\opEnd^{\catM}_M$ the associated dg operad.
Also let $\opC$ be a cooperad.
Then we have a dg Lie algebra 
\[
 \bigl(\Hom_{\frkS}(\opC,\opEnd^{\catM}_M), [\, ], \partial \bigr).
\]
\end{cor}

%%%%%%%%%%%%%%%%%%%%%%%%%%%%%%%%%%%%%%%%%%%%%%%%%%%%%%%%%%%%%%%%%%%%%%
%\subsubsection{The Maurer-Cartan equation and twisted morphisms}

%As before, we work on a dg pseudo-tensor category $\catM$ and
%a dg co-pseudo-tensor category $\catL$.
%Choose an object $(L,M) \in \catL \otimes_{\bbK} \catM$,
%and consider the dg operad $\opP_M$ and the dg cooperad $\opC_L$.
%Then we have a dg Lie algebra $(\opHom_{L,M},[\ ], \partial)$.

%Hereafter we use abbreviated symbols 
%$\opP := \opP_M$ and $\opC := \opC_L$
%if no confusion may occur.
For any dg Lie algebra 
$\g=\bigl(\g, [\, ], \partial\bigr)$ 
one can consider the Maurer-Cartan equation in it.
\[
 \partial(\alpha)+\dfrac{1}{2}[\alpha,\alpha]=0, \quad \alpha \in \g.
\]
In the next subsection we recall the meaning of the solution of 
this Maurer-Cartan equation in $\Hom_{\frkS}(\opC,\opP)$
following \cite[\S6.5]{LV}.
We close this subsection by  

\begin{dfn*}
We denote by $\Tw(\opC,\opP)$ 
the space of homogeneous solutions of 
the Maurer-Cartan equation of degree $-1$.
\[
 \Tw(\opC,\opP)
 := \bigl\{ \alpha \in \Hom_{\frkS}(\opC,\opP) \mid 
    |\alpha|=-1, \ \partial(\alpha)+\tfrac{1}{2}[\alpha,\alpha]=0 \bigr\}.
\]
We call its element a \emph{twisted morphism} 
from $\opC$ to $\opP$.
\end{dfn*}

%%%%%%%%%%%%%%%%%%%%%%%%%%%%%%%%%%%%%%%%%%%%%%%%%%%%%%%%%%%%%%%%%%%%%%
%\subsubsection{Twisted composite products}
%\label{sss:TCP}
%
%Assume that we are given a dg operad $\opP$, a dg cooperad $\opC$ 
%and a graded $\frkS$-module homomorphism $\alpha:\opC \to \opP$ of degree $-1$.
%We set the derivation $d^r_\alpha$ on $\opC \circ \opP$ by
%\[
% d^r_{\alpha} :
% \opC \circ \opP \xrightarrow{\ \Delta_{(1)} \circ \id_{\opP}\ }
% (\opC \ccone \opC) \circ \opP 
% \xrightarrow{\ (\id_{\opC} \ccone \alpha)\circ \id_{\opP} \ }
% (\opC \ccone \opP) \circ \opP  \simeq  \opC \circ(\opP;\opP \circ \opP)
% \xrightarrow{\ \id_{\opC} \circ (\id_{\opP};\gamma)\ }
% \opC \circ(\opP;\opP) \simeq \opC \circ \opP,
%\]
%and 
%\[
% d_\alpha := d_{\opC} \circ \id_{\opP} 
% + \id_{\opC} \circ' d_{\opP}+ d^r_{\alpha}.
%\]
%Then
%\[
% d_{\alpha}^2 = d^r_{\partial(\alpha)+\alpha * \alpha}
%\]
%so that $d_\alpha^2=0$ if and only if $\alpha \in \Tw(\opC,\opP)$. 
%
%Thus for $\alpha \in \Tw(\opC,\opP)$ we have a complex
%\begin{equation}\label{eq:CaP}
% \opC \circ_{\alpha} \opP := 
% (\opC \circ \opP, d_{\alpha}),
%\end{equation}
%which is called the (right) \emph{twisted composite product}.

%%%%%%%%%%%%%%%%%%%%%%%%%%%%%%%%%%%%%%%%%%%%%%%%%%%%%%%%%%%%%%%%%%%%%%
%%%%%%%%%%%%%%%%%%%%%%%%%%%%%%%%%%%%%%%%%%%%%%%%%%%%%%%%%%%%%%%%%%%%%%
\subsection{Koszul operads}
\label{subsec:Koszul}

%The goal of this subsection is Fact \ref{fct:Koszul}.

%%%%%%%%%%%%%%%%%%%%%%%%%%%%%%%%%%%%%%%%%%%%%%%%%%%%%%%%%%%%%%%%%%%%%%
\subsubsection{Cobar construction and twisting morphisms}
\label{sss:bar}
Let us recall the \emph{bar construction} of an augmented dg operad.

First we recall the free operad $\opF(M)$ and the cofree cooperad $\opF^c(M)$ 
of an $\frkS$-module $M$.
We refer \cite[\S5.6, \S5.8]{LV} for a full account.

Let $M$ be an $\frkS$-module.
Define inductively $\opT_n M$ by
\[
 \opT_0 M := I, \quad
 \opT_1 M := I \oplus M, \quad
 \opT_n M := I \oplus (M \circ \opT_{n-1} M)\ (n\ge2).  
\]
Using the inclusion $\opT_n M \hookrightarrow \opT_{n+1}$,
we get an $\frkS$-module $\opT M$ given by
\[
 \opT M := \cup_{n\ge0}\opT_n M = \colim_{n\ge0} \opT_n M.
\]
Then $\opT M$ has a structure of an operad 
such that any $\frkS$-module morphism $f:M \to \opP$ 
to an operad extends uniquely to an operad morphism 
$\opT M \to \opP$.
We will denote the resulting operad by $\opF(M)$.

A dual construction also exists.
Recall that a coaugmented cooperad $\opC=(\opC,\Delta,\ve,\eta)$ 
is a cooperad $(\opC,\Delta,\ve)$ 
with a cooperad morphism $\eta:I \to \opC$ 
such that $\ve \eta = \id_{I}$.
Here $I$ denotes the identity cooperad 
whose underlying $\frkS$-module is given by \eqref{eq:op:I}.
$\eta$ is called the coaugmented morphism of $\opC$.
The image of $1 \in I(1) = \bbK \subset I$ is denoted by 
$\id \in \opC(1)$.

The $\frkS$-module $\opT M$ explained above 
has a structure of cooperad
such that any $\frkS$-module morphism $g:\opC \to M$ 
sending $\id$ to $0$ from a conilpotent cooperad $\opC$
factors through $\opT M$ 
(we have the projection $\opT M \twoheadrightarrow M$).
%which is cofree in the category of conilpotent cooperads.
We will denote by $\opF^c(M)$ this universal cooperad,
and call it the \emph{cofree cooperad} of $M$.

%Next we introduce some notations for augmented operads.
%Let $\opP=(\opP,\gamma,\eta,\ve)$ be an augmented operad 
%with the augmentation morphism $\ve:\opP \to I$.
%Denote by $\overline{\opP} := \Ker(\ve)$ the augmentation ideal.
%Now for an augmented dg operad  $\opP$,
%define a conilpotent cooperad  $B\opP$ by 
%\[
% B \opP := \opF^c(s \overline{\opP}).
%\]
%Here $s$ is the shift of $\frkS$-module given by $(s M)_p := M_{p-1}$.
%$B \opP$ has a dg structure, 
%but we skip the explanation and refer \cite[\S6.5.1]{LV} for the full account.
%We  denote the resulting dg cooperad by the same symbol $B \opP$,
%and call it the \emph{bar construction}

Let us introduce the notation for the \emph{cobar construction}.
For a coaugmented dg cooperad $\opC=(\opC,\Delta,\ve,\eta)$, 
there is an augmented dg operad $\Omega \opC$
whose operad structure is given by
\[
 \Omega \opC := \opF(s^{-1} \overline{\opC}).
\]
Here $\overline{\opC}$ is the coaugmentation coideal 
$\Cok(\eta:I \to \opC)$ 
of the coaugmented cooperad $\opC$.
We skip the explanation of the dg structure and 
refer \cite[\S6.5]{LV} for the full account.
$\Omega \opC$ is called the \emph{cobar construction}.

Now we can state the following representability of the functor 
given by $\Tw$.

\begin{fct}[{\cite[Theorem 6.5.7]{LV}}]
\label{fct:OmTwB}
For every augmented dg operad $\opP$
and conilpotent dg cooperad $\opC$,
there exist natural isomorphisms
\[
 \Hom_{dg Op}(\Omega \opC,\opP)  \simeq 
 \Tw(\opC,\opP).
% \simeq
% \Hom_{dg coOp}(\opC, B\opP).
\]
Here $dg Op$ denotes %and $dg coOp$ denote 
the category of dg operads. %and of dg cooperads respectively.
\end{fct}

%%%%%%%%%%%%%%%%%%%%%%%%%%%%%%%%%%%%%%%%%%%%%%%%%%%%%%%%%%%%%%%%%%%%%%
\subsubsection{Koszul dual}
\label{sss:Koszul-dual}

Let us recall the Koszul dual cooperad %and the Koszul dual operad 
of a quadratic operad.
We follow the description in \cite[\S\S7.1--7.2]{LV}.

Recall the free operad $\opF(E)$ associated to an $\frkS$-module $E$.
$\opF(E)$ has a weight grading $w$ defined by
\[
 w(1)=0,\quad w(\mu)=1 \ (\mu \in E(n), n>0)
\]
where $1$ is the basis of the first part $I(1)=\bbK$ 
of the identity $\frkS$-module $I$,
and 
\[
w(\mu;\nu_1,\ldots,\nu_n)
:= w(\mu)+w(\nu_1)+\cdots+w(\nu_n)
\]
for $(\mu;\nu_1,\ldots,\nu_n) \in \opT_n E$.
Denote by $\opF(E)^{(d)}$ the weight $d$ subspace.
%For example, we have
%\begin{align*}
%&\opF(E)^{(0)} = I =(0,\bbK,0,0,\ldots),\\
%&\opF(E)^{(1)} = E =(0,E(1),E(2),E(3),\ldots),\\
%&\opF(E)^{(2)} = (0,E(1)^{\otimes 2},\ldots).
%\end{align*}

Let $(E,R)$ be a pair of a graded $\frkS$-module $E$ and 
a graded sub-$\frkS$-module $R \in \opF(E)^{(2)}$.
Such a pair is called a \emph{quadratic operad}.
The \emph{quadratic operad} associated to $(E,R)$ is defined to be 
\[
 \opP(E,R) := \opF(E)/(R)
\]
where $(R)$ is the (operadic) ideal generated by $R$.
In terms of the universality,
it is universal among the quotient operads $\opP$ of $\opF(E)$ 
such that the composite 
$R \hookrightarrow \opF(E) \twoheadrightarrow \opP$
vanishes.
$\opP(E,R)$ has a weight grading induced by that on $\opF(E)$,
and the weight $d$ subspace is denoted by $\opP(E,R)^{(d)}$.
%It looks like
%\begin{align*}
%&\opP(E,R)^{(0)} = I, \\
%&\opP(E,R)^{(1)} = E,\\
%&\opP(E,R)^{(2)} = (0,E(1)^{\otimes 2}/R(1),\ldots).
%\end{align*}

Similarly, the \emph{quadratic cooperad} $\opC(E,R)$ 
associated to $(E,R)$ 
is the sub-cooperad of the cofree cooperad $\opF^c(E)$ 
which is among the sub-cooperads $\opC$ of $\opF^c(E)$ 
such that the composite 
$\opC \hookrightarrow \opF^c(E) \twoheadrightarrow \opF^c(E)^{(2)}/(R)$ 
vanishes.
The underlying $\frkS$-module of $\opC(E,R)$ is $\opF E$,
and $\opC(E,R)$ has a weight grading.
The weight $d$ subspace is denoted by $\opP(E,R)^{(d)}$.
%We have
%\begin{align*}
%&\opC(E,R)^{(0)} = I, \\
%&\opC(E,R)^{(1)} = E,\\
%&\opC(E,R)^{(2)} = (0,R(1),R(2),R(3),\ldots).
%\end{align*}

Now for a quadratic operad $\opP = \opP(E,R)$,
its \emph{Koszul dual cooperad} is defined to be the quadratic cooperad 
\[
 \opP^{c!} := \opC(s E, s^2 R)
\]

\subsubsection{Koszul operad}

%For a quadratc data $(E,R)$, we obviously have
%$\opP(E,R)^{(1)} = E$ and $\opC(E,R)^{(1)} =  E$.
%Thus for a quadratic operad $\opP=\opP(E,R)$,
%we can defin the morphism $\kappa$ by
%\[
% \kappa: \opP^{c!} = \opC(s E,s^2 R)
% \longtwoheadrightarrow s E \xrightarrow{\ s^{-1} \ }
% E \longhookrightarrow \opP(E,R) = \opP.
%\]
%
%Now we have
%\begin{fct*}[{\cite[Lemma 7.4.1]{LV}}] 
%\[
% \kappa \in \Tw(\opP^{c!},\opP)
%\]
%\end{fct*}
%
%Then by \S\ref{sss:TCP}, we can consider the complex of $\frkS$-modules 
%\[
% \opP^{c!} \otimes_{\kappa} \opP
% =(\opP^{c!} \circ \opP, d_{\kappa}),
%\]
%called the \emph{Koszul complex} of the operad $\opP$.
%
%\begin{dfn*}
%A quadratic operad $\opP$ is called a \emph{Koszul operad} 
%if its Koszul complex $\opP^{c!} \otimes_{\kappa} \opP$ is acyclic.
%\end{dfn*}

A \emph{Koszul operad} is a quadratic operad $\opP$ 
whose Koszul complex is acyclic.
We will not give the definition of Koszul complex for a quadratic operad,
and refer \cite[\S7.4]{LV}.

%Although the fact we will use later is that $\opLie$ is a Koszul,
Let us mention following criterion of Koszul-ness.
By the definition of the cobar construction, 
for a quadratic operad $\opP=\opP(E,R)$ 
we have a natural projection
\[
 p: \Omega \opP^{c!} \longtwoheadrightarrow \opP. 
\]

\begin{fct}[{\cite[Theorem 7.4.2]{LV}}]
\label{fct:koszul}
A quadratic operad $\opP$ is Koszul if and only if 
the natural projection $p$
is a quasi-isomorphism of dg operads.
\end{fct}

%%%%%%%%%%%%%%%%%%%%%%%%%%%%%%%%%%%%%%%%%%%%%%%%%%%%%%%%%%%%%%%%%%%%%%
\subsection{Twisted morphism and algebra structure}
\label{subsec:mc-alg}

%Now let us go back to coisson and chiral algebras on a smooth curve $X$.
%Let us choose a pseudo-tensor structure on $\catM(X)$,
%such as $\catM(X)^{c}$ or $\catM(X)^{\ch}$.
%We simply denote the chosen pseudo-tensor category by $\catM$.

Fix a pseudo-tensor category by $\catM$.
%For an object $M \in \catM(X)$,
As before, for an object $M \in \catM$, 
we denote the associated operad by 
\[
 \opEnd_M \equiv \opEnd_M^{\catM}:= \oplus_{n \ge 1} P^{\catM}_n(M).
\]

For a Koszul operad $\opB$, a homotopy $\opB$-algebra,
or a $\opB_\infty$-algebra, in $\catM$ 
means an $\Omega \opB^{c!}$-algebra in $\catM$.
Thus a $\opB_\infty$-algebra structure on $M \in \catM$ 
is a morphism of dg operads
\[
 \Omega \opB^{c!} \longto \opEnd_M.
\]
A $\opB$-algebra is an example of $\opB_\infty$-algebra,
since the structure morphism gives
\[
  \Omega \opB^{c!} \xrightarrow{\ qis \ } \opB \longto \opEnd_M.
\]
Here $qis$ is the quasi-isomorphism given by Fact \ref{fct:koszul}.

By Fact \ref{fct:OmTwB} and Corollary \ref{cor:g_CM}, 
a morphism of dg operads $\Omega \opB^{c!} \to \opEnd_M$
is equivalent to a twisted morphism 
in the dg Lie algebra.

\begin{dfn}
\label{dfn:g_BM}
\[
 \g \equiv \g^{\catM}_{\opB,M} 
 := \Hom_{\frkS}\bigl(\opB^{c!},\opEnd^{\catM}_M\bigr).
\]
\end{dfn}

It has a weight grading $\g=\oplus_{n\ge0}\g^{(n)}$ 
given by
\[
 \g^{(n)} := \Hom_{\frkS}\bigl((\opB^{c!})^{(n)},\opEnd^{\catM}_M\bigr),
\]
where $\opB^{c!} = \oplus_{n\ge0} (\opB^{c!})^{(n)}$ 
denotes the weight decomposition of the Koszul dg cooperad $\opB^{c!}$.

Now writing down the differential of $\g$, we have

\begin{fct}[{\cite[Proposition 10.1.4]{LV}}]
\label{fct:Koszul}
For a Koszul operad $\opB$, 
a $\opB_\infty$-algebra in a pseudo-tensor category $\catM$ 
is a $\opB$-algebra
if and only if its twisting morphism is concentrated in weight $1$.
\end{fct}

Thus, as explained in the introduction, we have
\begin{align*}
\{ \text{ $\opB$-algebra structures on $M$} \} 
\ \stackrel{\ \ 1:1 \ \ }{\longleftarrow \joinrel \longrightarrow} \ 
 \MC(\g_{\opB,M}) := 
 \{ \text{weight $1$ elements in $\Tw(\g_{\opB,M})$} \}.
\end{align*}

%%%%%%%%%%%%%%%%%%%%%%%%%%%%%%%%%%%%%%%%%%%%%%%%%%%%%%%%%%%%%%%%%%%%%%
%%%%%%%%%%%%%%%%%%%%%%%%%%%%%%%%%%%%%%%%%%%%%%%%%%%%%%%%%%%%%%%%%%%%%%
\subsection{Deformation theory}
\label{subsec:2:deform}

%As we have seen in the previous subsection,
%for a Koszul operad $\opB$,
%a $\opB$-algebra structure on an object $M \in \catM(X)$ 
%is equivalent to a weight $1$ solution of the Maurer-Cartan equation 
%in the convolution dg Lie algebra 
%$\g_{\opB,M} = \Hom_{\frkS}\left(\opB^{c!},\opEnd(M)\right)$.

Let us briefly recall the deformation theory using dg Lie algebras
following \cite{GM}.
For a full account using operadic language, 
see \cite[\S12.2]{LV}.

For a dg Lie algebra $\g=(\g,[ \ ],d)$,
let us denote the space of solutions of the Maurer-Cartan equation 
with (weight) grading $1$ by $\MC(\g)$.
%Writing the graded structure by $\g=\oplus_{n}\g^n$,
%we have 
%\[
% \Tw(\g) = \{ \varphi \in \g^{-1} \mid 
% d(\varphi) + \tfrac{1}{2}[\varphi,\varphi]=0\}.
%\]
On this space the subspace $\g^0$ acts infinitesimally on $\MC(\g)$.
Namely the map
\[
 \g^0 \ni \lambda \longmapsto 
 d \lambda + [-,\lambda] \in \Gamma(\MC(\g),T\MC(\g))
\]
is a morphism of Lie algebras,
where the target is the Lie algebra of vector fields on $\MC(\g)$.
The exponentiated action can be written as 
\[
 e^{t \lambda} \alpha = e^{t \ad(\lambda)}(\alpha)
 + \dfrac{\id-e^{t \ad(\lambda)}}{\ad(\lambda)}(d \lambda).
\]
Let $G$ be the adjoint group of $\g^0$.
We denote by
\[
 \catMC(\g) := \MC(\g)/G
\]
the moduli space of Maurer-Cartan elements of $\g$,
which is considered as a coset, 
or a groupoid (called the \emph{Deligne groupoid}).

Let $\opB$ be a Koszul operad,
$M$ be an object of a pseudo-tensor category $\catM$.
Hereafter we write 
\[
 \g := \g^{\catM}_{\opB,M}
 = \bigl(\Hom_{\frkS}(\opB^{c!},\opEnd_M), [\,], \partial \bigr).
\]
A solution $\varphi \in \MC(\g^{\catM}_{\opB,M})$ 
gives a $\opB$-algebra structure on $M$.
Once $\varphi$ is chosen,
one can twist the dg Lie algebra $\g$ as

\begin{dfn}\label{dfn:twist-dgla}
For $\varphi \in \MC(\g)$,
the twisted dg Lie algebra $\g^\varphi$ is defined to be 
\[
 \g^\varphi := 
  \bigl(\Hom_{\frkS}(\opB^{c!},\opEnd_M),[\,],\partial_\varphi\bigr),
 \quad
 \partial_\varphi := \partial + [\varphi,-].
\]
\end{dfn}

Consider a local Artin $\bbK$-algebra $R$ with $\frkm$ the maximal ideal.
An \emph{$R$-deformation} of $\varphi$ is an $R$-linear 
$\opB \otimes_{\bbK} R$-algebra structure on $M \otimes_{\bbK} R$
which is equal to $\varphi$ modulo $\frkm$.
The \emph{equivalence} of $R$-deformations can be defined naturally,
and we denote it by $\sim$.
We define
\[
 \Def_{\varphi}(R) := \{\text{$R$-deformations of $\varphi$}\},\quad
 \catDef_{\varphi}(R) :=\Def_{\varphi}(R)/\sim.
\]
The latter one can be considered as a set or a groupoid.
The standard deformation theory says

\begin{fct*}[{\cite[Proposition 12.2.6]{LV}}]
We have a natural bijections (or equivalence of groupoids)
\[
 \Def_{\varphi}(R) \simeq \MC(\g^{\varphi} \otimes \frkm),\quad
 \catDef_{\varphi}(R) \simeq 
  \catMC(\g^{\varphi} \otimes \frkm).
\]
\end{fct*}

We also have the standard descriptions of rigidity and obstructions.
For $R=\bbK[[t]]$, the condition $H^1(\g^\varphi)=0$ implies 
that any $R$-deformation of $\varphi$ is trivial,
and the condition $H^2(\g^\varphi)=0$ implies that 
any infinitesimal deformation of $\g^\varphi$ extends to an $R$-deformation.

%The \emph{deformation functor} associated to $\g$ is defined to be 
%\[
% \Def_{\g}: (\text{local Artin $\bbK$-algebras}) \longto \catSet,
% \quad 
% R = \bbK \oplus \frkm \longmapsto \catMC(\g \otimes \frkm).
%\]

%Once a weight $1$ solution $\mu$ of the Maurer-Cartan equation is given,
%we have a twisted dg Lie algebra
%\[
%  \g^\mu_{\opP,V}
%  := (\Hom(\opP^{c!},\opEnd_V), [ \ ], d^\mu := d + [\mu,-]).
%\]
%and it encodes the deformation of $\opP$-algebra structure $\mu$.
%\begin{align*}
%&\MC^1(\g^\mu_{\opP,V})
%  \stackrel{\ 1:1 \ }{\longleftarrow \joinrel \longrightarrow} 
% \{\text{$\opP$-algebra structures deforming $\mu$} \}. 
%\end{align*}

%Deformation functors are invariant under quasi-isomorphisms of 
%dg Lie algebras.
%This statement is used 
%in Kontsevich's proof \cite{K}
%of the existence of the deformation quantization of 
%Poisson manifolds.

%%%%%%%%%%%%%%%%%%%%%%%%%%%%%%%%%%%%%%%%%%%%%%%%%%%%%%%%%%%%%%%%%%%%%%
%%%%%%%%%%%%%%%%%%%%%%%%%%%%%%%%%%%%%%%%%%%%%%%%%%%%%%%%%%%%%%%%%%%%%%
%%%%%%%%%%%%%%%%%%%%%%%%%%%%%%%%%%%%%%%%%%%%%%%%%%%%%%%%%%%%%%%%%%%%%%
\section{Chiral dg Lie algebra}
\label{sect:chdgla}

We now apply the construction of convolution dg Lie algebras 
in the operad theory \cite{LV} explained in \S\ref{sect:conv} to 
the chiral or coisson operad explained in \S\ref{sect:CA}.

%%%%%%%%%%%%%%%%%%%%%%%%%%%%%%%%%%%%%%%%%%%%%%%%%%%%%%%%%%%%%%%%%%%%%%
%%%%%%%%%%%%%%%%%%%%%%%%%%%%%%%%%%%%%%%%%%%%%%%%%%%%%%%%%%%%%%%%%%%%%%
\subsection{Description of our dg Lie algebras}
\label{subsec:calc}

Let $X$ be a smooth curve and $\catM(X)$ be the category of right 
$\shD$-modules over $X$.
Consider the pseudo-tensor categories
$\catM(X)^{c}$ and $\catM(X)^{\ch}$.
Applying Definition \ref{dfn:g_BM} of $\g^{\catM}_{\opB,M}$ 
to the case $\opB=\opLie$ and $\catM=\catM^{\ch},\catM^{c}$ 
we have following two Lie algebras. 

\begin{dfn}
For an object $M \in \catM(X)$, set 
\begin{align*}
\g^{\ch}_{M} := \g^{\catM^{\ch}}_{\opLie,M}
  = \Hom_{\frkS}\bigl(\opLie^{c!},\opEnd^{\catM^{\ch}}_M\bigr),
\quad
\g^{c}_{M} := \g^{\catM^{c}}_{\opLie,M}
  = \Hom_{\frkS}\bigl(\opLie^{c!},\opEnd^{\catM^{c}}_M\bigr).
\end{align*}
We call them the \emph{chiral} and \emph{coisson Lie algebras}
of $M$ respectively.
\end{dfn}

By the definitions of chiral and coisson algebras 
and the discussion in 
\S\ref{subsec:mc-alg}--\S\ref{subsec:deform},
these objects control the structures 
of chiral algebra and coisson algebra respectively.

Recall Definition \ref{dfn:dgHom} of the differential $\partial$.
For  $\g^{\ch}_M$ and $g^{c}_M$, 
since $\opLie^{c!}$, $\opEnd^{\ch}_M$ and $\opEnd^{c}_M$ 
have null differentials,
we have $\partial=0$ on $\g^{\ch}_M$ and $\g^{c}_M$.
Hence we called them (graded but with trivial differential) Lie algebras.

Let us write down these Lie algebras explicitly.
First we recall an explicit description of 
%the Koszul dual cooperad 
$\opLie^{c!}$.

\begin{lem}
\label{lem:Liec!}
The graded $\frkS$-module structure of $\opLie^{c!}$ is given by 
\[
 \opLie^{c!}(n) \simeq s^{1-n} \sgn_n,
\]
where 
%the basis $\ell_n$ has the grading $|\ell_n| = n-1$, and 
$\sgn_n$ is the sign representation of $\frkS_n$.
\end{lem}

\begin{proof}
%The operad $\opLie$ is the quadratic operad $\opP(E,R)=\opF(E)/(R)$,
%where $E=E(2)=\bbK \mu$ and $\frkS_2$ acts by the sign representation.
%For the description of $R \subset \opF(E)^{(2)}$,
%define
%\begin{align*}
%\mu_1 := \gamma(\mu;\mu,\id): \ &(x,y,z)\longmapsto \mu(\mu(x,y),z),\\
%\mu_2 := (\mu_1)^{(123)}:     \ &(x,y,z)\longmapsto \mu(\mu(z,x),y),\\
%\mu_3 := (\mu_1)^{(132)}:     \ &(x,y,z)\longmapsto \mu(\mu(y,z),x).
%\end{align*}
%These give a  basis of $\opF(E)^{(2)}$.
%Indeed, by definition the free operad and \eqref{eq:MoN(n)},
%\[
% \opF(E)^{(2)} = 
% (E \circ(I \oplus E))(3) \simeq 
% E \otimes \Ind^{\frkS_3}_{\frkS_2}E \simeq (E \otimes E)^{\oplus 3},
%\]
%so that $\dim \opF(E)^{(2)}=3$.
%Now we have  $R = \bbK(\mu_1+\mu_2+\mu_3) \in \opF(E)(3)$.
%Namely, $\mu$ presents the anti-symmetric bracket and 
%$R$ means the Jacobi rule.
%
%$\opLie^{c!}=\opC(s E,s^2R)$ 
%is isomorphic to the quadratic cooperad 
%$\opC(E,R)$ as $\frkS$-module (without grading),
%and by definition 
%\[
% \opC(E,R) = 
% \Ker\bigl(\opF^c(E) \longtwoheadrightarrow \opF^c(E)^{(2)}/R\bigr).
%\]
The shortest argument is to use the fact $\opLie^! \simeq \opCom$.
For a quadratic operad $\opP=\opP(E,R)$,
its  Koszul dual operad $\opP^!$ is defined to be 
\[
 \opP^{!} := (\opcoEnd_{s \bbK} \otimes_H \opP^{c!})^*.
\]
Here $\opcoEnd_{V}$ denotes the cooperad of endomorphisms on $V$ 
with $\opcoEnd_V=\oplus_{n \ge0} \Hom_{\bbK}(V,V^{\otimes n})$.
The Hadamard tensor product $\otimes_H$ of graded vector spaces is 
defined by $(M \otimes_H N)_n := M_n \otimes N_n$.
Finally $*$ denotes the linear dual.
As a corollary, we have
\[
 \opP^{c!} \simeq \opcoEnd_{s^{-1}\bbK}\otimes_H (\opP^!)^*
\]
Now for $\opP=\opLie$, 
since $\opP^!(n)=\opCom(n)$ is the trivial $\frkS_n$-module and 
$\opcoEnd_{s^{-1}\bbK}(n) \simeq s^{1-n} \sgn_n$ 
as graded $\frkS$-module 
(recall Remark \ref{rmk:sign} that 
 we are considering sign-graded complexes),
we have the result.
\end{proof}

Next recall that the chiral operad $\opEnd^{\ch}_M$ is given by
\[
 \opEnd^{\ch}_M(n)=
 \Hom_{\catM(X^n)}\bigl(j_*j^* M^{\boxtimes n},\Delta_* M\bigr)
\]
where $j:=j^{(n)}: X \hookrightarrow X^n$ is the diagonal embedding
and $\Delta:=\Delta^{(n)}: U^{(n)} \hookrightarrow X^n$ 
is the complement of diagonal divisors.
Then by Lemma \ref{lem:Liec!} we have
\begin{align*}
 \g^{\ch}_{M}(n) 
&\simeq
 s^{1-n} \sgn_n \otimes_{\frkS_n}
 \Hom_{\catM(X^n)}\bigl(j_*j^* M^{\boxtimes n},\Delta_* M\bigr)
 \simeq 
 s^{1-n}  \Hom_{\catM(X^n)}\bigl(j_*j^* (\wedge^n M),
  \Delta_* (M) \bigr).
\end{align*}
Let us restate this formula as 

\begin{lem}
\label{lem:gchM:Smod}
The graded $\frkS$-module structure of $\g^{\ch}_{M}$ is given by 
\begin{align*}
 \g^{\ch}_{M}(n) \simeq s^{1-n} C^{\ch,n}(M),
\quad
 C^{\ch,n}(M):= \Hom_{\catM(X^n)}\bigl(j_*j^* (\wedge^n M),
  \Delta_* (M) \bigr).
\end{align*}
\end{lem}

\begin{rmk}\label{rmk:CE}
$\g^{\ch}_{M}$ looks quite similar to the Chevalley-Eilenberg complex 
$C^\bullet(L,L) = \Hom(\wedge^\bullet L,L)$ of a Lie algebra $L$ over $\bbK$.
\end{rmk}

Next we study the pre-Lie structure $\star$ on $\g^{\ch}_{M}(n)$.
As expected from the above remark, 
the result has the same form as the Nijenhuis-Richardson product 
\cite[\S13.2.9]{LV}
on the Chevalley-Eilenberg complex.

For an operad $(\opP,\gamma,\eta)$, we denote by 
$\circ_i$ the $i$-th partial composition.
It is given by
\begin{align}\label{eq:circ_i}
 \mu \circ_i \nu := \gamma(\mu;\id,\ldots,\id,\nu,\id,\ldots,\id)
\end{align}
with $\nu$ sitting at the $i$-th position 
and $\id:=\eta(1) \in \opP(1)$ as before.
For $\mu \in \opP(m)$ and $\nu \in \opP(n)$,
$\mu \circ_i \nu$ is defined for $1 \le i \le m$ and 
$\mu \circ_i \nu \in \opP(m+n-1)$.

Let us also recall that a \emph{$(p,q)$-shuffle} is a permutation 
\[
 \sigma=
 \begin{pmatrix}
 1   & \cdots & p   & p+1 & \cdots & p+q \\
 i_1 & \cdots & i_p & j_1 & \cdots & j_q
 \end{pmatrix}
 \in \frkS_{p+q}
\]
such that $i_1<\cdots<i_p$ and $j_1<\cdots<j_q$.
The inverse of a $(p,q)$-shuffle is called \emph{$(p,q)$-unshuffle}.
Denote by $\frkS^{-1}_{p,q} \subset \frkS_{p+q}$ 
the subset of $(p,q)$-unshuffles.

\begin{lem*} 
For $f \in C^{\ch,p}(M)$ and $g \in C^{\ch,q}(M)$, we have 
\begin{align}\label{eq:NR}
 f \star g = 
 \sum_{\sigma \in \frkS^{-1}_{p,q}}
 \sgn(\sigma) (-1)^{(p-1)(q-1)} (f \circ_1 g)^\sigma.
\end{align}
Here $\circ_1$ is given by \eqref{eq:circ_i} 
with $\gamma=\gamma^{\ch}$ the composition on 
the chiral operad $\opEnd^{\ch}_M$.
\end{lem*}

\begin{proof}
Let us denote by $\widetilde{f} \in \g^{\ch}_M$ 
the element corresponding to $f$ 
under the isomorphism in Lemma \ref{lem:gchM:Smod},
and similarly by $\widetilde{g}$ the one corresponding to $g$.
Recall the expression \eqref{eq:star} of $\star$.
For $\mu \in \opLie^{c!}$ with $\Delta(\mu)$ given by \eqref{eq:decomp} 
we have
\begin{equation}\label{eq:star:Lie}
 (\widetilde{f} \star \widetilde{g})(\mu)
=\sum\sum_i 
 \gamma^{\ch}\bigl(f;\ve(\nu_1),\ldots,\ve(\nu_{i-1}),g(\nu_i),
               \ve(\nu_{i+1}),\ldots,\ve(\nu_n)\bigr).
\end{equation}
Note that $\Delta$ means the decomposition in the cooperad $\opLie^{c!}$,
and $\gamma^{\ch}$ is the composition in the operad $\opEnd^{\ch}_M$.
By the above expression, 
it is enough to consider the infinitesimal decomposition 
$\Delta_{(1)}$ (see Definition \ref{dfn:inf-comp:coop}).
By Lemma \ref{lem:Liec!} $\dim \opLie^{c!}(n)=1$ and 
denote the basis by $\ell_n$.
Also by the same lemma $\Delta$ is induced by 
the decompositions on $\opcoEnd$ and $\opCom^*$.
%$\frkS_{p,q}$ gives a section of 
%$\frkS_{p+q} \longtworightarrow \frkS_{p+q}/\frkS_p \times \frkS_q$.
Now one can find 
\begin{equation}\label{eq:Delta1:Lie}
 \Delta_{(1)}(\ell_n)=
 \sum_{\substack{p+q=n+1 \\ p,q>1}}
 \sum_{\sigma \in \frkS_{p,q}^{-1}}
 (-1)^{(p-1)(q-1)}(\ell_p;\ell_q,\id,\ldots,\id)^\sigma
\end{equation}
Here $\id=\eta(1) \in \opLie^{c!}(1)$ is the image of $1 \in \bbK$ 
under the coaugmentation $\eta$ of $\opLie^{c!}$.
Going back to $\star$, 
we note that $\widetilde{f}$ can be seen as the map $\ell_p \mapsto f$.
Then \eqref{eq:star:Lie} and \eqref{eq:Delta1:Lie} give the result.
\end{proof}

We summarize the argument so far in 

\begin{prop}
\label{prop:chiralLA}
The chiral Lie algebra $\g^{\ch}_{M}$ is described as 
\begin{align*}
 \g^{\ch}_{M}(n) \simeq s^{1-n} C^{\ch,n}(M),\quad
 C^{\ch,n}(M) := 
  \Hom_{\catM(X^n)}\bigl(j_* j^*(\wedge^n M), \Delta_* (M) \bigr).
\end{align*}
The Lie bracket is given by $[f,g]=f \star g - g \star f$,
where the pre-Lie product $\star$ is \eqref{eq:NR}
with $\circ_1$ corresponding to 
the composition map $\gamma=\gamma^{\ch}$   
of the chiral operad $\opEnd^{\ch}_M$.
\end{prop}

%The Lie bracket is the anti-symmetrization
%\[
% [f,g] := f \star g - g \star f,
%\]
%and it is an analogue of the Nijenhuis-Richardson bracket.

By the same argument we have

\begin{prop}
\label{prop:coissonLA}
The coisson Lie algebra $\g^c_{M}$ is described as 
\begin{align*}
 \g^{c}_{M}(n) \simeq s^{1-n} C^{c,n}(M),\quad
 C^{c,n}(M) := 
 \bigoplus_{S \in Q([n])}
 \Hom_{\catM(X^S)}\bigl(
  \wedge_{s \in S} \bigl( M^{\otimes^! [n]_s}\bigr), 
  \Delta^{(S)}_* (M) \bigr) \otimes (\otimes_{s \in S} \opLie_{[n]_s}),
\end{align*}
where $[n]=\{1,\ldots,n\}$ and 
$[n]_s := \{x \in [n] \mid \pi(x)=s \}$ 
for $\pi:[n] \twoheadrightarrow S$ and $s \in S$.
The Lie bracket is similarly described as in Proposition \ref{prop:chiralLA} 
where we replace $\gamma^{\ch}$ with the composition $\gamma^c$ of 
the cooperad $\g^{c}_{M}$.
\end{prop}

%%%%%%%%%%%%%%%%%%%%%%%%%%%%%%%%%%%%%%%%%%%%%%%%%%%%%%%%%%%%%%%%%%%%%%
%How does the complex $\g^c_{M}$ look like?
%
%As an $\frkS_n$-module,
%\begin{align*}
% \g^c_M(n)_n 
%&:= \Hom(\opLie^{c!},\opEnd^{c}_M(n)_n)
%\\
%& \simeq 
% \Hom(\opLie^{c!},\Hom_{\catM(X^n)}(M^{\boxtimes n},\Delta_* M))
%\\
%&\simeq
% \Hom_{\catM(X^n)}((s M)^{\boxtimes n}(\infty \Delta_{\text{diag}}),
%  s \Delta_*  M))
%\\
%&\simeq
% s^{1-n} \Hom_{\catM(X^n)}(\wedge^n M(\infty \Delta_{\text{diag}}),\Delta_* M).
%\end{align*}

Next assume that we are given 
$\mu \in \MC(\g^{\ch}_{M})$  or $\MC(\g^{c}_{M})$ 
\and consider the corresponding twisted dg Lie algebra 
(see Definition \ref{dfn:twist-dgla}).

\begin{dfn}
The twisted dg Lie algebra 
\[
 \g^{\ch,\mu}_{M} := 
 \bigl(\Hom_{\frkS}(\opLie^{c!},\opEnd^{\ch}_M), 
  [\, ], \partial_\mu \bigr),
 \quad
 \partial_\mu := \partial + [\mu,-] = [\mu,-].
\]
is called the \emph{chiral dg Lie algebra}.
Similarly,
\[
 \g^{c,\mu}_{M} := 
 \bigl(\Hom_{\frkS}(\opLie^{c!},\opEnd^{c}_M), 
  [\, ], \partial_\mu \bigr)
\]
is called the \emph{coisson dg Lie algebra}.
\end{dfn}

Recall that $\mu \in \MC(\g^{\ch}_{M}) \subset \g^{\ch}_M(2)$ is 
a binary operation.
Then from the description of $\g^{\ch}_M$, we find 

\begin{prop}
\label{prop:chDGLA:d}
The differential $\partial_\mu$ of $\g^{\ch,\mu}_M$ is described by 
\begin{align*}
 \partial_{\mu}(f)(x_1 \wedge \cdots \wedge x_{n+1}) = 
 &\sum_{i=0}^{n} (-1)^i 
  \mu\bigl(x_i,f(x_0 \wedge \cdots \wh{x}_i \cdots \wedge x_n)\bigr) \\
 &+\sum_{0\le i<j \le n} (-1)^{i+j-1} 
  f\bigl(\mu(x_i,x_j) \wedge x_0 \wedge \cdots \wh{x}_i \cdots \wh{x}_j
    \cdots \wedge x_n\bigr)
\end{align*}
for $f \in C^{\ch,n}(M) \simeq s^{n-1}\g^{\ch}_M(n)$
and $x_0 \wedge \cdots \wedge x_n \in j_* j^*(\wedge^{n+1} M)$.
$\wh{x}_i$ denotes skipping the term $x_i$.
The same expression holds for $g^{c,\mu}_M$ with $\mu \in \MC(\g^{c}_M)$.
\end{prop}

\begin{rmk*}
Continuing Remark \ref{rmk:CE},
the differential obtained coincides with that of 
the Chevalley-Eilenberg complex $C^\bullet(L,L)$.
Indeed, the convolution dg Lie algebra 
$(\Hom_{\frkS}(\opLie^{c!},\opEnd_L),[\,],\partial_\mu)$ 
with $\opEnd_L = \oplus_{n\ge1}\Hom_{\bbK}(L^{\otimes n},L)$
and $\mu$ the given Lie bracket on $L$
is nothing but $C^\bullet(L,L)$ up to shift,
as explained in \cite[\S13.2.7]{LV}.
\end{rmk*}

Let us close this section by translating our dg Lie algebra 
in the language of vertex algebras.
Recall Fact \ref{fct:VA=CA} of the correspondence between 
vertex and chiral algebras.
Let $\mu \in \MC(\g^{\ch}_{\shV^r})$ be the element 
corresponding to a quasi-conformal vertex algebra $(V,T,\ket{0},Y)$,
where $\shV^r$ is the right $\shD_X$-module attached to $V$.
Then the differential $\partial_\mu$ in Proposition \ref{prop:chDGLA:d}
reads 
\begin{equation}
\label{eq:DSK}
\begin{split}
 \partial_\mu f(a_0,\ldots,a_n)
 =&\sum_{r=0}^{n-1} (-1)^r Y(a_r,z) f(a_0,\ldots,\wh{a}_r,\ldots,a_n) \\
  &+\sum_{0\le r < s \le n-1}
  (-1)^{n+r+s} f\bigl(a_0,\ldots,\wh{a}_r,\ldots,\wh{a}_s,
  \ldots,a_n, Y(a_r,z) a_{s}\bigr)
\end{split}
\end{equation}
with $a_i \in \shV^r$ and $f \in C^{ch,n}(\shV^r)$ .

In the coisson dg Lie algebra for a coisson algebra structure 
$\mu \in \in \MC(\g^{\ch}_{\shV^r})$ corresponding to a vertex Poisson algebra 
$(V,Y_+,Y_-)$,
we have a similar formula as \eqref{eq:DSK} 
replacing $Y$ with $Y_{-}$.
%Looks similar to the Chevalley complex.

\begin{rmk*}
\begin{enumerate}
\item 
One can apply the construction of convolution dg Lie algebra to 
the $*$-pseudo tensor structure on $\catM(X)$
(see Definition \ref{dfn:*PRS}).
Namely, we replace $\opEnd^{ch}_M$ by 
the operad $\opEnd^{*}_M := \oplus_{n} P^*_n(M)$.
As explained in \cite[Chap.\ 19]{FBZ} and \cite{BD},
the corresponding algebra structure , called $\opLie^*$-algebra,
is equivalent to a vertex Lie algebra.
Thus the resulting dg Lie algebra $\g^{*,\mu}_M$ controls 
deformations of a vertex Lie algebra structure corresponding to $\mu$.
Lie conformal algebra cohomology \cite{DSK} 
by De Sole and Kac is a cohomology theory of vertex Lie algebras.
and from the expressions of our construction 
(or, simply from the similarity to the Chevalley-Eilenberg complex),
$\g^{*,\mu}_M$ coincides with ours.

\item
Tamarkin+'s dg Lie algebra for chiral algebras \cite{T}
seems to be almost equivalent to ours.
However it considers a pro-finite limit of pseudo-tensor structure.
At present we don't know the role of the pro-finite limit.

\item
Yi-Zhi Huang introduced in \cite{H1, H2} a cohomology theory 
for graded vertex algebras,
where a certain condition on convergence is required to 
the coefficients in the cohomology complex.
Except for this convergence problem,
his construction seems to coincide with ours.
Since our construction requires quasi-conformal property to 
vertex algebras, Huang's construction is not covered by ours.
\end{enumerate}
\end{rmk*}

%%%%%%%%%%%%%%%%%%%%%%%%%%%%%%%%%%%%%%%%%%%%%%%%%%%%%%%%%%%%%%%%%%%%%%
%%%%%%%%%%%%%%%%%%%%%%%%%%%%%%%%%%%%%%%%%%%%%%%%%%%%%%%%%%%%%%%%%%%%%%
\subsection{Deformation problem}
\label{subsec:deform}

Let us recall the sequence \eqref{eq:operad_seq} of operads.
On the binary part it yields
\begin{align}\label{eq:operad_seq:rep}
 \opEnd^{\ch}_M(2) \longtwoheadrightarrow
 \gr\opEnd^{\ch}_M(2) \longhookrightarrow 
 \opEnd^{c}_M(2).
\end{align}
The construction of convolution dg Lie algebra immediately implies

\begin{prop}
The above sequence induces the following morphism of dg Lie algebras 
preserving weight gradings.
\[
  \wh{\psi}: \g^{\ch}_M \longto \g^{c}_M.
\]
\end{prop}

Since the Maurer-Cartan equation is given universally for dg Lie algebras,
the existence of $\wh{\psi}$ yields

\begin{thm}
The morphism $\wh{\psi}$ induces a map
\[
 \psi: \MC(\g^{\ch}_M) \longto \MC(\g^{c}_M).
\]
\end{thm}

\begin{dfn}
We call $\mu \in \MC(\g^{\ch}_M)$ 
a \emph{chiral deformation quantization} of  $\mu^c \in \MC(\g^{c}_M)$ 
if $\psi(\mu)=\mu^c$.
\end{dfn}

\begin{rmk}\label{rmk:dq-cdq}
Here is the totally different feature of our quantization 
problem from the usual deformation quantization of associative algebras.
The dg Lie algebra of a Poisson algebra $(A,\circ,\{\,\})$ is 
given by $\Wedge^\bullet_A \Der(A)$ with the Gerstenhaber bracket.
We have the Hochschild-Kostant-Rosenberg quasi-isomorphism 
$f:\Wedge^\bullet_A \Der(A) \simto H(C^\bullet(A,A))$.
The difficult point is that the linear map
$\wh{f}:\Wedge^\bullet_A \Der(A) \to C^\bullet(A,A)$ 
obtained naturally from $f$ is not a morphism of dg Lie algebras.
\cite{K} succeeded to deform $\wh{f}$ to an $L_\infty$-morphism 
$f_\infty$, and obtained a dg Lie algebra morphism. 
In our situation, we have a natural morphism of dg Lie algebras from 
the beginning.
\end{rmk}

One may ask why we don't treat the coset 
$\catMC(\g)=\MC(\g)/G$ 
where $G$ is the adjoint group of $\g^0$.
By Propositions \ref{prop:chiralLA} and \ref{prop:coissonLA} 
we have $\g^0=\g(1)=\Hom_{\catM(X)}(M,M)$
for $\g=\g^{\ch}_M$ and $\g^c_M$,
so that it is enough to treat $\MC(\g)$.

Obviously, if the map  $\psi$ is surjective,
then a chiral deformation quantization exists.

\begin{prop}
If $M$ is a projective $\shD$-module, then $\psi$ is surjective, 
so that for any $\mu^c \in \MC(\g^{c}_M)$ has a 
a chiral deformation quantization.
\end{prop}

\begin{proof}
If $M$ is projective then 
%the morphism 
$\gr\opEnd^{\ch}_M \to \opEnd^{c}_M$
in \eqref{eq:operad_seq:rep}
is an isomorphism, as remarked in \cite[\S3.2.4]{BD}.
\end{proof}

Now our main theorem is 

\begin{thm}
$\psi$ is always an injection.
\end{thm}

\begin{proof}
By Propositions \ref{prop:chiralLA} and \ref{prop:coissonLA},
we have 
\[
 \MC(\g^{\ch}_M) \simeq 
 \{ \alpha \in C^{\ch,2}(M) \mid \alpha \star \alpha =0\},
 \quad
 \MC(\g^{\ch}_M) \simeq 
 \{ \alpha \in C^{c,2}(M) \mid \alpha \star \alpha =0\} 
\]
with 
\begin{align*}
&C^{\ch,2}(M) = 
 \Hom_{\catM(X^2)}\bigl(j_*j^*(\wedge^2 M),\Delta_*(M)\bigr)
\\
&C^{c,2}(M) := 
 \Hom_{\catM(X^2)}\bigl(\wedge^2 M, \Delta_* (M) \bigr) 
 \bigoplus
 \Hom_{\catM(X)}\bigl(M^{\otimes^! 2}, 
 M \bigr) \otimes \opLie(2).
\end{align*}
Now recall the special filtration $W^\bullet$ on $\opEnd^{ch}_M$ 
given by \S\ref{sss:cl}, \eqref{eq:ch:filt}.
On $C^{\ch,2}$ it yields a decreasing filtration 
\[
  0 = W^0 C^{\ch,2} \subset  W^{-1} C^{\ch,2} \subset W^{-2} C^{\ch,2}
 \subset W^{-3} C^{\ch,2} = C^{\ch,2}.
\]
We want to write it down this sequence explicitly.
Recall that we have an explicit description \eqref{eq:spfilt:grd}
of the graded components $\gr^{W}_{\bullet}j_*j^* \omega_{X^I}$.
For $I=[2]$ it reads
\begin{align*}
&j_*j^* \omega_{X^2}
 =W^0 j_*j^* \omega_{X^2} \supset 
  W^{-1} j_*j^* \omega_{X^2} \supset
  W^{-2} j_*j^* \omega_{X^2} \supset  
  W^{-3} j_*j^* \omega_{X^2} =0,
\\
&\gr^{W}_{0}=0, \quad
 \gr^{W}_{-1}=\Delta_* \omega_X, \quad
 \gr^{W}_{-2}=\omega_X^{\boxtimes 2} \otimes \opLie(2)^*.
\end{align*}
We also have the following Cousin complex.
\[
 0 \longto \omega_{X}^{\boxtimes 2} \longto 
 j_*j^* \omega_X  \longto \Delta_* \omega_X \longto 0.
\]
Thus the exact sequence splits.
Now going back to $W^\bullet C^{\ch,2}$,
we find
\[
 C^{\ch,2} \simeq \gr_{W}^\bullet C^{\ch,2} \simeq \gr \opEnd_M(2). 
\]
Thus 
the first arrow in \eqref{eq:operad_seq:rep} 
is an isomorphism,
so that \eqref{eq:operad_seq:rep} is an injection in total.
Thus the induced map $\psi$ is an injection.
\end{proof}

\begin{cor}\label{cor:unique}
A chiral deformation quantization is unique if it exists.
\end{cor}

\begin{rmk*}
As shown in \cite[\S2.6]{BD} and mentioned in the introduction,
we have two standard examples of coisson algebras.
The fist one corresponds to 
the vertex Poisson algebra $V_{\infty}(\g)$ arising from 
the affine vertex algebra,
and the second one $W_{\infty}(\g,e_{\reg})$ arising from the $W$-algebra.
By Corollary \ref{cor:unique}, 
their chiral deformation quantizations are unique.
We know the existence, namely $V_{k}(\g)$ and $W_{k}(\g,e_{\reg})$,
so that all the chiral deformation quantizations are isomorphic 
to these standard vertex algebras.
\end{rmk*}

\subsection*{Acknowledgement}
The author is supported by the Grant-in-aid for 
Scientific Research (No.\ 16K17570), JSPS.
This work is also supported by the 
JSPS for Advancing Strategic International Networks to 
Accelerate the Circulation of Talented Researchers
``Mathematical Science of Symmetry, Topology and Moduli, 
  Evolution of International Research Network based on OCAMI"''.

The author would like to express special gratitude to Professor T.~Kuwabara
whose talk at Kyoto University in the summer 2015 invoked 
the motivation of this note,
and for pointing out the literature \cite{H1, H2}.
The author would also like to thank the organizers of 
\emph{Algebraic Lie Theory and Representation Theory 2016}
where a part of this study is presented.

Large part of this note is written during the author's stay at UC Davis 
in the spring 2016. 
The author would like to thank the institute for support and hospitality,
and Professor M.~Mulase for the discussion around opers.

%%%%%%%%%%%%%%%%%%%%%%%%%%%%%%%%%%%%%%%%%%%%%%%%%%%%%%%%%%%%%%%%%%%%%%
%%%%%%%%%%%%%%%%%%%%%%%%%%%%%%%%%%%%%%%%%%%%%%%%%%%%%%%%%%%%%%%%%%%%%%
%%%%%%%%%%%%%%%%%%%%%%%%%%%%%%%%%%%%%%%%%%%%%%%%%%%%%%%%%%%%%%%%%%%%%%


\begin{thebibliography}{MMNN}

\bibitem[BD04]{BD}
Beilinson,~A., Drinfeld,~V.,
\emph{Chiral algebras},
American Mathematical Society Colloquium Publications, \textbf{51}, 
American Mathematical Society, Providence, RI, 2004.


\bibitem[B86]{B}
Borcherds,~R., 
\emph{Vertex algebras, Kac-Moody algebras, and the Monster}, 
Proc.\ Nat.\ Acad.\ Sci.\ U.S.A., Vol.\ 
\textbf{83} (1986), No. 10,  3068--3071.

\bibitem[DSK09]{DSK}
De Sole,~A., Kac,~V., 
\emph{Lie conformal algebra cohomology and the variational complex},
Comm.\ Math.\ Phys.\ \textbf{292} (2009), no.\ 3, 667--719.

\bibitem[FBZ04]{FBZ}
Frenkel,~E., Ben-Zvi,~D.,
\emph{Vertex algebras and algebraic curves}, Second edition, 
Mathematical Surveys and Monographs, \textbf{88}, 
American Mathematical Society, Providence, RI, 2004. 

\bibitem[GM88]{GM}
Goldman.~W.~M., Millson,~J.~J.,  
\emph{The deformation theory of representations of fundamental groups of 
compact K\"{a}hler manifolds}, 
Inst.\ Hautes \'{E}tudes Sci.\ Publ.\ Math.\ (1988), no.\ 67, 43--96.


\bibitem[H14a]{H1} Huang, Y.,
\emph{First and second cohomologies of grading-restricted vertex algebras}, 
Comm.\ Math.\ Phys.\ \textbf{327} (2014), no.\ 1, 261--278. 

\bibitem[H14b]{H2} Huang, Y.,
\emph{A cohomology theory of grading-restricted vertex algebras}, 
Comm.\ Math.\ Phys.\ \textbf{327} (2014), no.\ 1, 279--307. 

\bibitem[K03]{K}
Kontsevich,~M.,
\emph{Deformation quantization of Poisson manifolds},
Lett.\ Math.\ Phys.\ \textbf{66} (2003), no.\ 3, 157--216.


\bibitem[LV12]{LV}
Loday,~J., Vallette,~B., 
\emph{Algebraic operads}, 
Grundlehren der Mathematischen Wissenschaften, 
\textbf{346}. Springer, Heidelberg, 2012.

\bibitem[M71]{M}
MacLane,~S., 
\emph{Categories for the working mathematician},
Graduate Texts in Mathematics, \textbf{5}, 
Springer-Verlag, New York, 1971.


\bibitem[T02]{T}
Tamarkin,~D.,
\emph{Deformations of chiral algebras}, 
Proceedings of the International Congress of Mathematicians, Vol. II 
(Beijing, 2002), 105--116, 
Higher Ed. Press, Beijing, 2002. 
\end{thebibliography}
\end{document}